\documentclass[11pt, twoside]{amsart}
\usepackage[cp1251]{inputenc}
\usepackage{color}
\usepackage{amsthm,amssymb,latexsym}
\usepackage[centertags]{amsmath}
\usepackage{tikz}
\usetikzlibrary{arrows, positioning}
\usetikzlibrary{calc}

\usepackage[T2A]{fontenc}
\usepackage{MnSymbol}
\usepackage{bbm}
\usepackage{centernot}
\usepackage{mathrsfs}
\usepackage{amsthm}
\usepackage{amsfonts}
\usepackage{mathtools}
\usepackage{textcomp}
\usepackage{newlfont}
\usepackage{enumitem}  
\usepackage[normalem]{ulem} 
\usepackage{cancel}
\usepackage{graphs} 
\usepackage[colorlinks=true,urlcolor=blue,pdfborder={0 0 0}]{hyperref}
\hypersetup{linkcolor=[rgb]{0,0,0.6}}
\hypersetup{citecolor=[rgb]{0,0.6,0}}

\newif\ifPDF
\ifx\pdfoutput\undefined\PDFfalse
\else \ifnum \pdfoutput > 0 \PDFtrue
        \else \PDFfalse
        \fi
\fi

\usepackage{tkz-graph}
\tikzset{EdgeStyle/.style = {->}}
\tikzset{LabelStyle/.style= {fill=yellow}}
\usetikzlibrary{shapes,snakes,calendar,matrix,backgrounds,folding}

\usepackage{bbm}

\usepackage[top=1in, bottom=1.25in, left=1.4in, right=1.4in]
{geometry}

\theoremstyle{definition}
\newtheorem{definition}{Definition}[section]
\theoremstyle{remark}
\newtheorem{example}[definition]{Example}

\newtheorem{remark}[definition]{Remark}

\theoremstyle{plain}
\newtheorem{thm}[definition]{Theorem}
\newtheorem{theorem}[definition]{Theorem}
\newtheorem{prop}[definition]{Proposition}
\newtheorem{lemma}[definition]{Lemma}

\newtheorem{corol}[definition]{Corollary}
\newtheorem{claim}[definition]{Claim}

\def\ol{\overline}
\def\ov{\overline}
\def\tl{\widetilde}
\def\wh{\widehat}

\newcommand{\N}{\mathbb N}
\newcommand{\Z}{\mathbb Z}

\newcommand{\be}{\begin{equation}}
\newcommand{\ee}{\end{equation}}
\newcommand{\ba}{\begin{aligned}}
\newcommand{\ea}{\end{aligned}}
\newcommand{\mc}{\mathcal}

\newcommand{\De}{\Delta}

\numberwithin{equation}{section}
\newcommand{\ignore}[1]{}

\begin{document}

\title{Horizontally stationary generalized Bratteli diagrams}

\author{Sergey Bezuglyi}
\address{Department of Mathematics,
University of Iowa, Iowa City, IA 52242-1419
USA}
\email{sergii-bezuglyi@uiowa.edu}

\author{Palle E.T. Jorgensen}
\address{Department of Mathematics, University of Iowa, Iowa City, IA 52242-1419 USA}
\email{palle-jorgensen@uiowa.edu}

\author{Olena Karpel}
\address{AGH University of Krakow, Faculty of Applied Mathematics, al. Adama Mickiewicza~30, 30-059 Krak\'ow, Poland \&
B. Verkin Institute for Low Temperature Physics and Engineering,
47~Nauky Ave., Kharkiv, 61103, Ukraine}

\email{okarpel@agh.edu.pl}

\author{Jan Kwiatkowski}

\address{Faculty of Mathematics and Computer Science, Nicolaus Copernicus University, ul. Chopina 12/18, 87-100 Toru\'n, Poland}

\email{jkwiat@mat.umk.pl}

\date{}

\begin{abstract}
Bratteli diagrams with countably infinite levels exhibit a new phenomenon: they can be horizontally stationary. The incidence matrices of these horizontally stationary Bratteli diagrams are infinite banded Toeplitz matrices. In this paper, we study the fundamental properties of horizontally stationary Bratteli diagrams. In these diagrams, we provide an explicit description of ergodic tail invariant probability measures.
For a certain class of horizontally stationary Bratteli diagrams, we prove that all ergodic tail invariant probability measures are extensions of measures from odometers. Additionally, we establish conditions for the existence of a continuous Vershik map on the path space of a horizontally stationary Bratteli diagram.

\end{abstract}

\maketitle

\section{Introduction}

In the present paper, we define and study the class of 
\textit{horizontally stationary generalized
Bratteli diagrams}. These diagrams can only be defined when all levels contain infinitely many vertices.
We begin by recalling the main concepts used throughout the paper.
A \textit{Bratteli diagram} refers to a graded graph $B=(V,E)$
where the sets of vertices and edges are partitioned into disjoint
unions: $V=\bigsqcup_{n\geq0}V_{n}$ and $E=\bigsqcup_{n\geq0}E_{n}$, 
with $E_{n}$ representing the edges between the vertices from 
$V_{n}$
and $V_{n+1}$. The condition $|V_{n}|<\infty$ for all $n$ defines
the class of \textit{standard} Bratteli diagrams which have played a crucial
role in the classification of AF $C^{*}$-algebras \cite{Bratteli1972} 
and, later, in the classification of minimal homeomorphisms of 
a Cantor set \cite{HermanPutnamSkau1992, GiordanoPutnamSkau1995}. 
The diagrams, where 
$V_{n}$ is countably infinite, are called 
 \textit{generalized Bratteli diagrams}; they were introduced and studied in recent works \cite{BezuglyiDooleyKwiatkowski2006, BezuglyiJorgensen2022, BezuglyiJorgensenKarpelSanadhya2023, BezuglyiKarpelKwiatkowski2024, BezuglyiKarpelKwiatkowskiWata,
BezuglyiJorgensenSanadhya2024}.  
For such diagrams, we will assume that $V_n = \Z$ for all $n$.
The sets $E_{n}$ determine the sequence of incidence
matrices $F_{n}, n \in \N_0,$ with entries that indicate the number of edges connecting
vertices of levels $V_{n}$ and $V_{n+1}$. 

We stress that the Bratteli diagrams introduced initially by Bratteli in 
\cite{Bratteli1972} entail the following finiteness assumption: it is assumed that every level-set in the diagram is a specified finite set. However, for our present context, it will be important to focus instead on the wider setting, i.e., the present case when each level in the diagram is a countably infinite set. This general setting allows us, in turn, to define the property of ``horizontally stationary'' (see below) which is our present focus. Moreover, this horizontal stationarity property fits the needs of important new applications. It also allows us to attack problems (for these generalized Bratteli diagrams) where the choice of incidence matrices in the vertical direction is non-stationary, i.e., when the incidence matrices of consecutive level sets change in the forward time direction. We further stress that problems in the literature dealing with the dynamics described by vertically non-stationary diagrams are difficult.

Bratteli diagrams are widely used in dynamics to construct models for transformations and study their properties. We refer to 
\cite{HermanPutnamSkau1992, GiordanoPutnamSkau1995, GlasnerWeiss1995, BezuglyiJorgensen2015, BrJoOstr2004, 
GPS2010,  Putnam2018, BezuglyiKarpel2016, DurandPerrin2022, 
BezuglyiKarpel_2020} and the 
literature therein.

Every homeomorphism of a Cantor set and every aperiodic
Borel automorphism of a standard Borel space can be realized as a transformation
(called \textit{Vershik map}) acting on the path space of a (generalized)
Bratteli diagram \cite{HermanPutnamSkau1992, Medynets2006, DownarowiczKarpel_2019, BezuglyiDooleyKwiatkowski2006}. 
This approach has significantly impacted the development of Cantor and Borel dynamics, as the properties of Vershik maps are more transparent and easier to study.

To define a horizontally stationary generalized Bratteli diagram, we 
consider the following 
property: the set of edges $e \in E_n$ incoming to a vertex $v$ does not depend on $v$. This implies that the sets $r^{-1}(v)$ and 
$r^{-1}(v')$ are identical up to a horizontal shift for any vertices $v$ and $v'$ from the same level. A generalized Bratteli
diagram $B$ satisfying this property is called \textit{horizontally stationary}; further details are given in Section \ref{sect 3}. In terms of incidence matrices, this means that $F_n$
is a banded matrix with equal entries along every diagonal, i.e., 
$F_n$ is a Toeplitz matrix.  

Our main results begin with the observation that the structure of horizontally stationary Bratteli diagrams provides an effective framework for studying tail invariant measures and the Vershik map (or tail equivalence relation $\mathcal R$). In Section \ref{sect 2}, we revisit the fundamental definitions related to Bratteli diagrams and examine key features of incidence matrices, such as equal row and equal column sums. Section \ref{sect 3} deals with the properties of horizontally stationary generalized Bratteli diagrams that directly follow from the definitions. Additionally, we apply the Fourier transform to reformulate the criterion for a sequence of positive vectors $(p^{(n)})$ 
 to determine a tail-invariant measure on the path space of a Bratteli diagram, see Proposition \ref{prop inv meas}. Section \ref{sect 4} focuses on the study of
tail invariant measures on path spaces $X_B$ of horizontally stationary 
Bratteli diagrams. 
Every sequence of vertices $\ol i = (i_n)$, $i_n \in V_n$, such that for every $n \in \mathbb{N}_0$ there exist edges between $i_n \in V_n$ and $i_{n+1} \in V_{n+1}$, defines a subdiagram $B(\ol i)$ of $B$.
The unique tail invariant ergodic probability  measure $\mu_{\ol i}$ 
on the path space $X_{B(\ol i)}$ can be extended by tail invariance to the ergodic measure 
$\wh \mu_{\ol i}$ supported by $\mathcal R$-saturation of
$X_{B(\ol i)}$. We provide the necessary and sufficient condition under which the extension $\wh \mu_{\ol i}$ is a finite measure. This occurs when the entry 
$f_{i_{n+1}i_n}^{(n)}$ dominates the sum $\sigma^{(n)}$ of all other entries in the $i_{n+1}$-row such that 
$$
\sum_{n = 0}^{\infty} \frac{\sigma^{(n)}}{f^{(n)}_{i_{n+1}i_{n}}} < \infty, 
$$
see Theorem \ref{Thm:ext_from_odom}. We then address the question of whether all ergodic probability measures on $X_B$
can be obtained through this construction. 
Theorem \ref{thm all erg meas} affirms this for a certain class of horizontally stationary Bratteli diagrams. 
The case when such Bratteli diagrams do not support finite measures is considered in Proposition \ref{prop:no_inv_meas}, which is further 
illustrated by examples. The final section considers orders on 
horizontally stationary Bratteli diagrams. We provide, in 
Theorem \ref{Thm:continVmap}, 
the necessary and sufficient conditions on an order under which the corresponding
Vershik map is a homeomorphism.

\section{Dynamics and measures on generalized Bratteli diagrams}
\label{sect 2}

This section briefly recalls the main definitions and notations regarding generalized Bratteli diagrams. For more details see \cite{BezuglyiJorgensenSanadhya2024, BezuglyiJorgensenKarpelSanadhya2023, BezuglyiKarpelKwiatkowski2024, 
BezuglyiKarpelKwiatkowskiWata}. 

\begin{definition}\label{Def:generalized_BD} A 
\textit{generalized Bratteli diagram} is a graded graph 
$B = (V, E)$ such that the 
vertex set $V$ and the edge set $E$ are disjoint unions  
 $V = \bigsqcup_{i=0}^\infty  V_i$ and $E = 
\bigsqcup_{i=0}^\infty  E_i$ of levels such that 

(i) the number of vertices at each level 
$V_i$, $i 
\in \N_0$, is countably infinite (in this paper, we will usually
identify each 
$V_i$ with $\Z$),

(ii) for every edge $e\in E$, we define the range and 
source maps $r$ and $s$ such that $r(E_i) = V_{i+1}$ and 
$s(E_i) = V_{i}$ for $i \in \N_0$, 

(iii) for every vertex $v \in V \setminus V_0$, we 
have $|r^{-1}(v)| < \infty$ ($|\cdot |$ denotes the cardinality of a set).
\end{definition} 

A finite or infinite sequence of edges $(e_i: e_i\in E_i)$ such 
that $s(e_i)=r(e_{i-1})$ is called a finite or infinite 
path, respectively. The set of infinite
paths starting at $V_0$ is denoted by $X_B$ and is called the \textit{path space} of the diagram $B$. For a finite path $\ol e = (e_0, ... , e_n)$, the set 
$$
    [\ol e] := \{x = (x_i) \in X_B : x_0 = e_0, ..., x_n = e_n\} 
$$ 
is the \textit{cylinder set} associated with $\ol e$. The {topology} on the path space $X_B$ generated by cylinder sets makes $X_B$ a 
zero-dimensional Polish space which is not locally compact, in general.

\begin{definition} Let $B = (V,E)$ be a generalized Bratteli diagram. 
    
(i) The set $V_i$ is called the \textit{$i$th level} of the diagram $B$.

(ii) For a vertex $v \in V_m$ and a vertex $w \in V_{n}$, let
$E(v, w)$ be the set of all finite paths between $v$ and $w$. 
Set $f^{(i)}_{v,w} = |E(v, w)|$ for every $w \in V_i$ and $v \in 
V_{i+1}$. This defines a sequence of non-negative countably infinite matrices $(F_i)$, $i \in \N_0$ 
(they are called the \textit{incidence  matrices}) where 
\begin{equation}\label{Notation:f^i}
    F_i = (f^{(i)}_{v,w} : v \in V_{i+1}, w\in V_i),\ \   
    f^{(i)}_{v,w}  \in \N_0.
\end{equation} 

(iii) Let $(F_n)$ be incidence matrices of $B$. If  $F_n = F$ for every $n \in 
\N_0$, then the diagram $B$ is called  \textit{(vertically) stationary}.

\end{definition}

For $w \in V_n$ and $n \in \N_0$, denote 
$X_w^{(n)} := \{x = (x_i)\in X_B : s(x_{n}) = w\}$. 
The collection $(X_w^{(n)}: w \in V_n)$ forms a partition 
of $X_B$ into clopen sets (\textit{Kakutani-Rokhlin towers})
corresponding to the vertices from $V_{n}$. For $w \in V_n$, set 
$$
H^{(n)}_w = \sum_{v_0 \in V_0} |E(v_0, w)|
$$ 
and $H^{(0)}_w = 1$ for all $w\in V_0$.
We call $H^{(n)}_w$ the \textit{height of the tower} $X_w^{(n)}$. 
The vectors of heights $H^{(n)} =  (H^{(n)}_w : w \in V_n)$ are 
related in the following obvious way:
 \begin{equation}\label{formula for heights}
F_n H^{(n)} = H^{(n+1)}, \ \ n \geq 1, \ \ \mbox{ and } \ \  H^{(n)} = F_{n-1} \ \cdots \ F_0 H^{(0)}.
\end{equation}

In this paper, we will also used the procedure of \textit{telescoping} of a generalized Bratteli diagram: given a generalized 
Bratteli diagram $B = (V,E)$ and a  monotone increasing sequence 
$(n_k : k \in \N_0), n_0 = 0$, define a generalized Bratteli 
diagram $B' = (V', E')$, where the vertex sets are 
determined by $V'_k = V_{n_k}$, and the edge sets  $E'_k = 
E_{n_{k}} \circ ...\circ E_{n_{k+1}-1}$ are formed by finite paths 
between the levels $V'_k$ and 
$V'_{k+1}$. The diagram $B' = (V', E')$  is called a 
\textit{telescoping} of the diagram $B = (V,E)$.

We will use below the sequence of \textit{stochastic incidence 
matrices} $(\tl F_n)$ associated to each generalized Bratteli diagram. Set
$\tl F_n = (\tl f_{vw}^{(n)} : v \in V_{n+1}, w \in V_n)$, where 
\be\label{eq_stoch matrix F_n}
\tl f_{vw}^{(n)} =  {f}_{vw}^{(n)}\ \cdot \frac{H_{w}^{(n)}}
{H_{v}^{(n + 1)}}.
\ee
Then we get from \eqref{formula for heights} that 
\begin{equation*}
\sum_{w\epsilon V_{n}}\tl f_{vw}^{(n)} =1, \quad v \in\ 
V_{n + 1}.    
\end{equation*}

For each $n \in \N_0$ and $m\in \N$, the product of incidence 
matrices $F_{n+m-1}\ \cdots \ F_n$ is denoted by ${G'}^{(n,m)}$. 
Then 
$ G^{(n,m)} = \tl F_{n+m-1} \ \cdots \ \tl F_n$ 
is the stochastic matrix corresponding to ${G'}^{(n,m)}$.

\begin{definition}
A matrix $F = (f_{ij})$ satisfies the \textit{equal row sum} property (we write $F\in ERS$ or $F\in ERS(r)$) if there exists $r$ such that $\sum_{j} f_{ij} = r$ for all $i$. A matrix $F = (f_{ij})$ has the {\em equal column sum} property ($F \in ECS$ or $F\in ERS(c)$) if there exists $c$ such that $\sum_{i} f_{ij} = c$ for all $j$.    
\end{definition}

\begin{remark}
Standard Bratteli diagrams with the $ERS$ property are models for Toeplitz subshifts (see~\cite{GjerdeJohansen2000}).    
\end{remark}

\begin{definition}\label{Def:BD_bdd_size} A generalized Bratteli 
diagram $B(F_n)$ is called of \textit{bounded size} if there exists 
a sequence of pairs of natural numbers $(t_n, L_n)_{n \in \N_0}$ 
such that, for all $n \in \mathbb{N}_0$ and all $v \in V_{n+1}$,
\begin{equation}\label{eq: Bndd size}
s(r^{-1}(v)) \in \{v - t_n, \ldots, v + t_n\} \quad \mbox{and} 
\quad \sum_{w \in V_{n}} f^{(n)}_{vw} = \sum_{w \in V_{n}} |E(w,v)| 
\leq L_n.
\end{equation} 
If the sequence $(t_n, L_n)_{n \in \N_0}$ can be chosen to be constant, i.e. $t_n = 
t$ and $L_n = L$ for all $n \in \N_0$, then we say that the 
diagram $B(F_n)$ is of \textit{uniformly bounded size}.
\end{definition} 

We will use the following convention for bounded size Bratteli 
diagrams. For each $n \in \N_0$, the pair of natural numbers 
$(t_n, L_n)$ are chosen to be the minimal possible.
We use the term \textit{measure} for a non-atomic positive Borel measure. 

\begin{definition}\label{Def:Tail_equiv_relation}
(i) Let $B$ be a standard or generalized Bratteli diagram.
Two paths $x= (x_i)$ and $y=(y_i)$ in $X_B$ are called 
\textit{tail equivalent} if there exists  $n \in \mathbb{N}_0$ 
such that $x_i = y_i$ for all $i \geq n$. This defines a 
countable Borel equivalence relation $\mathcal R$ on $X_B$ called the \textit{tail equivalence relation}.

(ii) A measure $\mu$ on $X_B$ is called \textit{tail invariant} if, for any $n \in \mathbb{N}$ and any cylinder sets $[\ol e] = [(e_0, \ldots, e_n)]$ and $[\ol e']= [(e'_0, \ldots, e'_n)]$ such that $r(e_n) = r(e'_n)$, we have $\mu([\ol e]) = \mu([\ol e'])$.
\end{definition}

To define a Vershik map on $X_B$, we need to introduce a linear order $\omega$ on each (finite) set $r^{-1}(v),\ v
\in V\setminus V_0$. This order defines a partial order $\omega$ on 
the sets  $E_i,\ i=0,1,...$, where edges
$e,e'$ are comparable if and only if $r(e)=r(e')$, see 
\cite{HermanPutnamSkau1992, BezuglyiKwiatkowskiYassawi2014}. 

\begin{definition}\label{Def:Ordered_GBD}
A generalized Bratteli diagram $B=(V,E)$
together with a partial order $\omega$ on $E$ is called 
\textit{an ordered generalized Bratteli diagram} $B=(V,E, \omega)$. 
\end{definition}

A (finite or infinite) path $\ov e= (e_i)$ is called 
\textit{maximal (minimal)} if every
$e_i$ is maximal (minimal) among all 
elements from $r^{-1}(r(e_i))$. Let $X_{\max}$ ($X_{\min}$) be the sets of all infinite maximal (minimal) paths in $X_B$. The sets $X_{\max}$ and $X_{\min}$ are closed sets that are always nonempty for standard Bratteli diagrams, but they can be empty for generalized Bratteli diagrams (see \cite{HermanPutnamSkau1992}, \cite{BezuglyiDooleyKwiatkowski2006} and \cite{BezuglyiJorgensenKarpelSanadhya2023}).

\begin{definition}\label{Def:VershikMap}  For an ordered 
generalized Bratteli diagram $B=(V,E, \omega)$, define a Borel 
transformation 
$\varphi_B \colon X_B \setminus X_{\max} \rightarrow X_B \setminus X_{\min}$
as follows.  Given $x = (x_0, x_1,...)\in X_B\setminus X_{\max}$, 
let $m$ be the smallest number such that $x_m$ is not maximal. Let 
$g_m$ be the successor of $x_m$ in the finite set $r^{-1}(r(x_m))$.
Set $\varphi_B(x)= (g_0, g_1,...,g_{m-1},g_m,x_{m+1},...)$
where $(g_0, g_1,..., g_{m-1})$ is the unique minimal path which starts at $V_0$ and ends in $s(g_{m})$. If  
$\varphi_B$ admits a bijective Borel extension 
to the 
entire path space $X_B$, then we call the Borel transformation 
$\varphi_B : X_B  \rightarrow X_B$ a \textit{Vershik map}, and 
the 
Borel dynamical system $(X_B,\varphi_B)$ is called a generalized 
\textit{Bratteli-Vershik} system.

\end{definition}


\section{Horizontally stationary generalized Bratteli diagrams}
\label{sect 3}
In this section, we introduce horizontally stationary 
generalized Bratteli diagrams and discuss their properties.  

\subsection{Structure of horizontally stationary Bratteli diagrams}

\begin{definition}
 Let $B = (V, E)$ be a generalized Bratteli diagram defined by the
 sequence of incidence matrices $(F_n)$. Suppose that 
 the following properties hold: for every $n\in \N_0$, we have

(i)  the vertices of $V_n$ are identified with $\Z$;

(ii) for every $i \in V_{n+1}$ and $j \in V_n$, the equality 
$f_{i,j}^{(n)} = f_{i+1,j+1}^{(n)}$ holds.

\noindent
We call such a Bratteli diagram \textit{horizontally stationary}.
\end{definition}

 Note that a horizontally stationary generalized Bratteli diagram is not necessarily stationary in the usual sense, i.e., the incidence matrices $F_n$ may be different from level to level. Remark also that the notion of a horizontally stationary Bratteli diagram does not apply to standard Bratteli diagrams with finite levels. 

One can easily verify that horizontally stationary generalized Bratteli diagrams possess the following properties.

\begin{prop}\label{Prop:prop_hor_gbd}
    Let $B = B(F_n)$ be a horizontally stationary generalized Bratteli diagram. Then for every $n \in \mathbb{N}_0$, the following statements hold:
    \begin{enumerate}
        \item for every $k,i,j \in \mathbb{Z}$, we have $f_{i,j}^{(n)} = f_{i+k,j+k}^{(n)}$,
    \item\label{prop2_hor_gbd} for every $i,j \in \mathbb{Z}$, the entry $f_{i,j}^{(n)}$ is a function of $j - i$,
    \item the matrix $F_n$ is banded,
    \item the sets $r^{-1}(i)$ and $s^{-1}(j)$ are translation equivariant, that is the geometrical structure of 
    these sets does not depend on $i\in V_{n+1}$ and $j\in V_n$,
    \item for every $n \in \N$, $i\in V_n$, and $j\in V_{n+1}$,
    $|r^{-1}(j)| = |s^{-1}(i)|$.
    \end{enumerate} 
 
\end{prop}

   The property~(\ref{prop2_hor_gbd}) of Proposition~\ref{Prop:prop_hor_gbd} means that, for all $n$, the incidence matrix $F_n$ has equal elements on diagonals, i.e.
   $$
  F_n = \begin{pmatrix}
   \ddots &  \ddots &  \ddots & \vdots & \vdots & \vdots & \udots\\
    \ldots & b_{-1}^{(n)} & b_{0}^{(n)} & b_{1}^{(n)} & \ldots\ & \ldots & \ldots\\
    \ldots & \ldots &  b_{-1}^{(n)} & b_{0}^{(n)} & b_{1}^{(n)} & \ldots & \ldots\\
    \ldots & \ldots &  \ldots & b_{-1}^{(n)} & b_{0}^{(n)} & b_{1}^{(n)} & \ldots\\
    \udots &  \vdots &  \vdots & \vdots & \ddots & \ddots & \ddots\\
  \end{pmatrix}, 
   $$
where $b_k^{(n)} = f_{i,j}^{(n)}$ for $k = j - i$.

\begin{remark}
Fix $n \in \mathbb{N}_0$ and let $i \in V_{n+1}$. Let $j_1 \in V_n$ be the source of the leftmost edge and $j_2\in V_n$ be 
the source of the rightmost edge in the set $r^{-1}(i)$. 
Observe that
$j_2 \geq j_1$ and
$j_2 - j_1$ does not depend on $i$. This means that the width of the band for $F_n$ is $j_2 - j_1 + 1$. Throughout the paper, we will be interested in the non-degenerate case $j_2 > j_1$.
Let $p = i - j_1$ and $q = j_2 - i$.
In general, $p\neq q$. If $p=q$ (or $i - j_1 = j_2 - i$), then $B$ is 
a generalized Bratteli diagram of bounded size with parameters 
$t_n = i - j_1$ and $L_n = |r^{-1}(i)|$. Then we can represent $F_n$ in the form
\be\label{eq3:Toeplitz m-x}
  F_n = \begin{pmatrix}
   \ddots &  \vdots &  \ddots & \vdots & \ddots & \vdots & \vdots & \vdots & \udots\\
   \ldots  & b_{-p}^{(n)} & \ldots & b_{0}^{(n)} & \ldots & b_{q}^{(n)} & 0 & 0 & \ldots\\
    \ldots & 0 & b_{-p}^{(n)} & \ldots & b_{0}^{(n)} & \ldots & b_{q}^{(n)} & 0 & \ldots\\
    \ldots & 0 & 0 & b_{-p}^{(n)} & \ldots & b_{0}^{(n)} & \ldots & b_{q}^{(n)} & \ldots\\
    \udots &  \vdots &  \vdots & \vdots & \ddots & \vdots & \ddots & \vdots & \ddots\\
  \end{pmatrix}.
\ee
\end{remark}

Matrices satisfying \eqref{eq3:Toeplitz m-x} are called \textit{infinite Toeplitz} matrices.

\begin{lemma}
Let $B= B(F_n)$ be a horizontally stationary generalized Bratteli diagram. Then the incidence matrices $F_n$ satisfy the equal row sum
$ERS(r_n)$ property and the equal column sum $ECS(c_n)$ property. 
Moreover, $r_n = c_n$ where $r_n$ and $c_n$ are the sums of rows and
columns, respectively.  
\end{lemma}

This statement is obvious and follows from \eqref{eq3:Toeplitz m-x} 
because 
\begin{equation}\label{eq:ERS+ECS}
r_n  = \sum_{i = -p}^{q} b_i^{(n)} = c_n.
\end{equation} 

\begin{remark}
We observe that if every incidence matrix $F_n$ of a Bratteli diagram $B$ has the ERS property, then this does not mean that $B(F_n)$ is
horizontally stationary. 
In fact, we note that the class of horizontally stationary generalized Bratteli diagrams is a proper subset of the set of Bratteli diagrams $B$ 
such that $B \in ECS \cap ERS$. To see this, we provide the following example, see Figure \ref{Fig:ERSECSnotHS}.

\begin{figure}
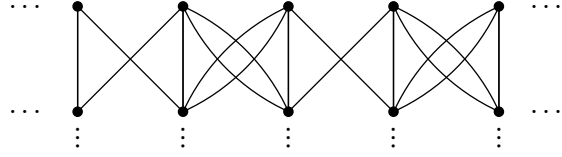

\unitlength=0,7cm
\begin{graph}(11,4)
 \graphnodesize{0.2}
 \roundnode{V11}(2,3)
 \roundnode{V12}(4,3)
 \roundnode{V13}(6,3)
 \roundnode{V14}(8,3)
  \roundnode{V15}(10,3)
 \roundnode{V21}(2,1)
 \roundnode{V22}(4,1)
 \roundnode{V23}(6,1)
  \roundnode{V24}(8,1)
    \roundnode{V25}(10,1)
  %
 %
 \graphlinewidth{0.025}

 \edge{V21}{V11}
  \edge{V21}{V11}
    \edge{V22}{V11}
    \edge{V21}{V12}
     
 \edge{V22}{V12}
 \edge{V22}{V12}

 \edge{V23}{V13}
  \edge{V23}{V13}

 \bow{V23}{V12}{0.09}
    \bow{V22}{V13}{-0.09}
    
 \bow{V23}{V12}{-0.09}
    \bow{V22}{V13}{0.09}
    
     \edge{V24}{V14}
  \edge{V24}{V14}
  
   \edge{V24}{V13}
    \edge{V23}{V14}
    
         \edge{V25}{V15}
  \edge{V25}{V15}
  
   \bow{V25}{V14}{0.09}
      \bow{V25}{V14}{-0.09}
    \bow{V24}{V15}{0.09}
        \bow{V24}{V15}{-0.09}

 \freetext(10.9,3){$\ldots$}
  \freetext(10.9,1){$\ldots$}
   \freetext(1,3){$\ldots$}
  \freetext(1,1){$\ldots$}
  \freetext(2,0.5){$\vdots$}
  \freetext(4,0.5){$\vdots$}
    \freetext(6,0.5){$\vdots$}
      \freetext(8,0.5){$\vdots$}  
        \freetext(10,0.5){$\vdots$} 
    
\end{graph}

\caption{Stationary generalized Bratteli diagram $B$ in $ERS(4)$ and $ECS(4)$ which is not horizontally stationary.}\label{Fig:ERSECSnotHS}
\end{figure}

\end{remark}

\begin{remark}\label{rem: heights}
(1) The product of two infinite Toeplitz matrices, $F$ and $F'$, is Toeplitz again. 
Moreover, if $F \in ERS(r)$ and $F' \in ERS(r')$, then $FF' \in 
ERS(rr')$.

(2) Let $B= B(F_n)$ be a horizontally stationary generalized Bratteli diagram with $F_n \in ERS(r_n)$.  Then, for every $j \in V_n$, 
\be\label{eq3:product r_i}
H_j^{(n)} = r_0 \ \cdots \ r_{n - 1}.
\ee
\end{remark}

Using Remark \ref{rem: heights}, we can
define the stochastic matrix $\tl F_n$ for a horizontally stationary 
$B = B(F_n)$:
\be\label{eq3:stoch}
\tl f^{(n)}_{ij} = f^{(n)}_{ij} \frac{H_j^{(n)}}{H_i^{(n+1)}} = f^{(n)}_{ij} \frac{1}{r_n}.
\ee
Note that both $F_n$ and $\tl F_n$ are Toeplitz matrices.

We give an easily verified criterion of horizontal stationarity of 
a generalized Bratteli diagram. 

Let $T =(t_{ij})$ be the infinite matrix such that 
$$
t_{ij} = \begin{cases}
1, & j= i+1, \ i \in \Z\\
0, & \mbox{otherwise}.
\end{cases}
$$
Let $x = (x_k)$. Then $T(x) = y$ where $y_k = x_{k+1}$.

\begin{lemma}
A generalized Bratteli diagram $B = B(F_n)$ is horizontally stationary 
if and only if $F_n T = TF_n$ for all $n \in \N_0$.
\end{lemma}

\begin{proof}
Straightforward.
\end{proof}

\begin{remark}\label{rem parallel}
Let $B$ be a horizontally stationary generalized Bratteli diagram
with the vertices indexed by $\Z$ at every level. Denote by 
$\tau$ the shift on $\Z$: $\tau (i) = i+1$, $i\in \Z$. Then $\tau$
generates a transformation acting on the set of all edges of $B$:
$\tau(e)$ is the edge such that $s(\tau (e)) = \tau (s(e))$ and 
$r(\tau (e)) = \tau (r(e))$. 
Two edges, $e$ and $f$ from $E_n$, are called \textit{parallel }
if there exists $k \in Z$ such that $s(f) = \tau^k(s(e))$ 
and  $r(f) = \tau^k(r(e))$. Two paths, $\ol x = (x_n)$ and $\ol y =
(y_n)$, are called parallel if $x_n$ is parallel to
$y_n$ for every $n$.
\end{remark}

We finish this subsection by  
discussing the connection between horizontally stationary 
and bounded size Bratteli diagrams.

\begin{definition}
    We say that a generalized Bratteli diagram of bounded size is \textit{full} if for every $n \geq 0$ and every $i \in V_{n+1}$ we have
    $$
   s(r^{-1}(i)) = [- t_n + i, t_n + i],
    $$
    which means that there are edges between $i \in V_{n+1}$ and every vertex from $\{i - t_n, \ldots, i + t_n\} \subset V_n$.
\end{definition}

Thus, every horizontally stationary Bratteli diagram is an edge subdiagram of a full generalized Bratteli diagram of bounded size (the notion of
edge and vertex subdiagrams can be found in \cite{BezuglyiKarpelKwiatkowski2024}, for example).

\begin{remark}
(i) One can check that the product of two Toeplitz matrices is a Toeplitz matrix. Indeed, let $F_n$ and $F_{n+1}$ be Toeplitz. Then for every $i,j \in \mathbb{Z}$, we have
$$
(F_{n+1}F_n)_{ij} = \sum_{k \in \mathbb{Z}} f^{(n+1)}_{ik}f^{(n)}_{kj} = \sum_{k \in \mathbb{Z}} f^{(n+1)}_{i+1,k+1}f^{(n)}_{k+1,j+1} =  (F_{n+1}F_n)_{i+1,j+1}.
$$ 

(ii) If $B$ is a horizontally stationary generalized Bratteli diagram then for every $i, j \in \mathbb{Z}$ we have
$$
s(r^{-1}(i + 1)) = s(r^{-1}(i)) + 1, \quad r(s^{-1}(i + 1)) = r(s^{-1}(i)) + 1,
$$
where for $S \subset \mathbb{Z}$, by $S + 1$ we mean $\{s + 1 : s \in S\}$.

(iii) If for every $n \geq 0$ and every $i \in \mathbb{Z}$ there exists $k \in \mathbb{Z}$ such that 
$$
f^{(n+1)}_{ik}f^{(n)}_{ki} > 0
$$
then after telescoping with respect to even levels for every vertex $i$ there are infinite vertical paths passing through $i$.
\end{remark}

\begin{remark}
    In \cite{BezuglyiJorgensenKarpelSanadhya2023}, the notion of isomorphism of generalized Bratteli diagrams was discussed (see e.g. \cite{Durand2010} for standard Bratteli diagrams). Isomorphism preserves many properties of Bratteli diagrams such as the set of tail invariant measures.     For any sequence $\{k_n\}_{n = 0}^{\infty} \subset \mathbb{Z}$, the property of horizontal stationarity is preserved under isomorphisms which shift vertices of each level, $g_n(w) = w + k_n$, and change correspondingly edges and their ordering. In particular, after such isomorphism, we can always assume that the diagram has vertical infinite maximal paths.
\end{remark}

\subsection{On tail invariant measures}

Let $B = B(F_n)$ be a horizontally stationary generalized Bratteli diagram. According to \eqref{eq3:Toeplitz m-x}, we can briefly 
write the incidence matrix $F_n$ as the doubly infinite sequence
$ b^{(n)} = [..., 0 , b^{(n)}_{-p}, \ldots , b^{(n)}_{q},0 ...]$ 
where $p = p(n)$, $q = q(n)$. 

More generally, we will consider infinite
sequences $ \alpha = (\alpha_k)_{k \in \Z}$ of real positive numbers. 
Take the formal Fourier series corresponding to $\alpha$:
\be\label{eq3:Fourier}
\wh \alpha(t):= \sum_{n\in \Z} \alpha_n e^{int}, \ t \in [-\pi, \pi].
\ee

If $\alpha$ has only finitely many non-zero entries, then the series in \eqref{eq3:Fourier} is, in fact, 
a trigonometric polynomial. If $\alpha \in \ell^1$, then the series 
in \eqref{eq3:Fourier} converges for all $t$. 

We note that the map $\alpha \mapsto \wh \alpha$ sends infinite 
vectors to scalar functions. Moreover, there is 
a one-to-one correspondence between function $\wh\alpha(t) \in 
L^2[-\pi, \pi]$ and sequences $\alpha$ of coefficients in 
the series \eqref{eq3:Fourier}.

Suppose that $F = (f_{ij})$ is a doubly infinite Toeplitz matrix whose diagonals
are formed by the numbers $(\alpha_k)$, i.e.,  $f_{ij} = \alpha_{i-j} $.
Then, for $x = (x_i)$, we have 
\be\label{eq3 convolution}
(F x)_i = \sum_{j} f_{ij} x_j = \sum_j \alpha_{i-j} x_j = 
(\alpha \star x)_i
\ee
where $\star$ denotes the convolution operation. 

\begin{remark}\label{rem F and conv}
It is well-known and can be easily checked that 
  \eqref{eq3:Fourier} and \eqref{eq3 convolution} imply the equality
$$
\wh{(\alpha \star \beta)}(t) = \wh{\alpha}(t) \cdot \wh{\beta}(t).
$$

Moreover, the converse is also true. Let $\wh\alpha$ and $\wh\beta$
be two functions from $L^2[-\pi, \pi]$ and $\alpha = (\alpha_k)$,
$\beta = (\beta_k)$ the corresponding sequences of 
Fourier coefficients. Then the sequence $\alpha \star \beta$ corresponds to the function 
$\wh\alpha \cdot \wh\beta$. 

\end{remark}

Apply Remark \ref{rem F and conv} to the case of a horizontally stationary 
generalized Bratteli diagram $B(F_n)$. We recall the following result
about tail invariant measures on the path space of a Bratteli diagram (see \cite[Theorem 2.9]{BezuglyiKwiatkowskiMedynetsSolomyak2010}, \cite[Theorem 2.3.2]{BezuglyiJorgensen2022}).

\begin{thm}\label{thm:mu_as_inv_lim}\label{BKMS_measures=invlimits}
Consider a Bratteli diagram (generalized or classical) $B = (V,E)$ 
 with the sequence of incidence matrices $(F_n)$. The following statements
 hold:
\begin{enumerate}

\item Let  $\mu$ be a tail invariant measure on $B$ which takes finite values on all cylinder sets. Define the sequences of vectors $p^{(n)} = \langle p^{(n)}_w : w \in V_n \rangle$, where 
\begin{equation}\label{eq:def_p_n}
    p^{(n)}_w= \mu([\ov e]),\ \  r(\ol e) = w, \ \ w\in V_n. 
\end{equation} 
Then the vectors $ p^{(n)}$ satisfy the relations 
\begin{equation}\label{eq:formula_p_n}
(F_n)^{T} p^{(n+1)} = p^{(n)},  \quad n\geq 0, 
\end{equation}

\item Suppose that $\{ p^{(n)}= (p_w^{(n)}) 
\}_{n \in \mathbb{N}_0}$ is a sequence of non-negative vectors 
satisfying \eqref{eq:formula_p_n}.
Then there exists a uniquely determined tail invariant measure $\mu$ such that $ \mu([\ov e]) = p_w^{(n)}$ for 
$w\in V_n$, and every path $\ol e$ ending at $w$, $n \in \mathbb{N}_0$.
\end{enumerate}
\end{thm}

Let $F_n^T = A_n$. Using \eqref{eq3 convolution}, we rewrite \eqref{eq:formula_p_n}
in the form 
$$
(A_n  p^{(n+1)})_i =  (a^{(n)} \star  p^{(n+1)})_i = p^{(n)}_i.
$$
Here the finite sequence $a^{(n)}$ is formed by the numbers 
$a^{(n)}_k$
such that $a^{(n)}_k = b^{(n)}_{-k}$, see \eqref{eq3:Toeplitz m-x}.

The following statement is a version of Theorem \ref{BKMS_measures=invlimits}. 

\begin{prop} \label{prop inv meas}
Let $B$ be a horizontally stationary generalized Bratteli diagram, and 
the matrices $A_n$ are defined by their entries on 
the diagonals $a^{(n)} = (a^{(n)}_k)_k$.
Let $\wh a^{(n)}$ and $\wh p^{(n)}$ be the Fourier transforms 
corresponding to the sequences $a^{(n)}$ and $ p^{(n)}$, $n \geq 0$.
The collection of infinite vectors $(p^{(n)})$ defines a tail invariant measure $\mu$ if and only if the functions 
$\wh p^{(n)}(t)$ satisfy the relations 
$$
\wh a^{(n)}(t)  \wh p^{(n+1)}(t) =
\wh p^{(n)}(t)
$$
for all $n \geq 0$.

\end{prop}

\begin{proof}
This result follows immediately from Theorem \ref{BKMS_measures=invlimits}, Remark \ref{rem F and conv}, and 
\eqref{eq3 convolution}
\end{proof}

\section{Ergodic probability tail invariant measures}\label{sect 4}

In this section, we discuss the measure extension from odometers and ECS subdiagrams for horizontally stationary generalized Bratteli diagrams. We describe explicitly the set of all 
ergodic probability tail invariant measures for horizontally stationary generalized Bratteli diagrams that belong to the class $\mathcal C$ (see \ref{eq4: 3-diag inc matr}). We also study Markov measures and horizontally invariant measures on horizontally stationary generalized Bratteli diagrams.

\subsection{Measure extension from an odometer for horizontally 
stationary diagrams}

Let $B$ be a horizontally stationary generalized Bratteli diagram with the incidence matrices $F_n = (f^{(n)}_{ij})$ of the form
 $$
  F_n = \begin{pmatrix}
   \ddots &  \vdots &  \ddots & \ddots & \ddots & \vdots & \ddots & \vdots & \vdots & \vdots & \udots\\
   \ldots  & b_{-m}^{(n)} & \ldots & b_{-1}^{(n)} & a_n & b_{1}^{(n)} & \ldots & b_{k}^{(n)} & 0 & 0 & \ldots\\
    \ldots & 0 & b_{-m}^{(n)} & \ldots & b_{-1}^{(n)} & a_n & b_{1}^{(n)} & \ldots & b_{k}^{(n)} & 0 & \ldots\\
    \ldots & 0 & 0 & b_{-m}^{(n)} & \ldots & b_{-1}^{(n)} & a_n & b_{1}^{(n)}  & \ldots & b_{k}^{(n)} & \ldots\\
    \udots &  \vdots &  \vdots & \vdots & \ddots & \vdots & \ddots & \ddots & \ddots & \vdots & \ddots\\
  \end{pmatrix}, 
   $$
where $k = k (n)$ and $m = m(n)$, and $k,m > 0$. We write here $a_n = b_{0}^{(n)}$. 
Let $\ov i = (i_n)$ be a sequence of integers such that for every $n \in \mathbb{N}_0$ there exist edges between $i_n \in V_n$ and $i_{n+1} \in V_{n+1}$. Denote by $B(\ov i)$ the odometer which is a vertex subdiagram of $B$ supported by vertices from $\ov i$. Let $\mu_{\ov i}$ be a unique probability tail invariant measure on the path space $X_{B(\ov i)}$. Then the measure $\mu_{\ov i}$ is defined by its values on cylinder sets
$$
\mu_{\ov i}([\ov e^{(n)}]) = \frac{1}{f^{(0)}_{i_1i_0} \ \cdots \ f^{(n-1)}_{i_ni_{n-1}}},
$$
where $[\ov e^{(n)}]$ is a cylinder set generated by a finite path $\ov e^{(n)}$ in $\ov B(\ov i)$ which ends at the vertex $i_n$ on level $n$. 
As was discussed in our earlier papers, such a measure admits 
a natural extension by tail invariance onto the smallest tail 
invariant set $\wh X_{B(\ov i)} = \mathcal R(X_{B(\ov i)})$ containing $X_{B(\ov i)}$. In other words, we extend the measure from a
section of the countable Borel equivalence relation $\mathcal{R}$ to its saturation. Details can be found, for example, in 
\cite{AdamskaBKK2017}, \cite{BezuglyiKarpelKwiatkowski2024}, \cite{Kechris2024}. 
The measure $\mu_{\ov i}$ is automatically extended to a measure 
$\wh{\mu}_{\ov i}$ on $\wh X_{B(\ov i)}$. 

We will answer here the following principal \textit{question}: under what conditions is the extension $\wh{\mu}_{\ov i}(\wh X_{B(\ov i)})$ finite? 

\begin{theorem}\label{Thm:ext_from_odom}
Let $B(F_n)$, $B(\ov i)$ and $\mu_{\ov i}$ be as above, and let 
$$
\sigma^{(n)} = \sum_{j\neq i_n}f^{(n)}_{i_{n+1},j}.
$$
Then 
\be\label{eq4:ext from odom}
\wh{\mu}_{\ov i}(\wh X_{B(\ov i)}) < \infty \ \ \Longleftrightarrow \ \ 
\sum_{n = 0}^{\infty} \frac{\sigma^{(n)}}{f^{(n)}_{i_{n+1}i_{n}}} < \infty. 
\ee
\end{theorem}

\begin{proof} In general settings, if $\ol B$ is a vertex subdiagram of a generalized Bratteli diagram $B$ that is built on the sets $(W_n)$, $W_n \subset V_n$, and $\ol\mu$ is a tail invariant measure on the path space
$X_{\ol B}$, then the extended measure $\wh{\ov \mu}$ can be found by the following formula
$$
\wh{\ov \mu}(\wh X_{\ov B}) = 1 + 
\sum_{n\geq 0} \sum_{v\in W_{n+1}} \sum_{w \in V_n \setminus W_n}
f^{(n)}_{vw} H^{(n)}_w \ol\mu([\ol e^{(n+1)}(v) ]).
$$
The reader can see this and similar results in \cite{AdamskaBKK2017},
\cite{BezuglyiKarpelKwiatkowskiWata}. 
Applying the above formula to the case when $B = B(F_n)$ and $\ol B = B(\ov i)$, we obtain
$$
\wh{\mu}_{\ov i}(\wh X_{B(\ov i)}) = 1 + \sum_{n \geq 0} \sum_{j \in V_n \setminus \{i_n\}} f_{i_{n+1},j}^{(n)}\frac{1}{f^{(0)}_{i_1i_0} \ \cdots \ f^{(n)}_{i_{n+1}i_{n}}}H_j^{(n)}.
$$
Since $B$ has the ERS property, we use \eqref{eq3:product r_i} to find that 
$$ 
H_j^{(n)} = r_0 \ \cdots \ r_{n-1},
$$
where the row sum $r_l$ of the matrix $F_l$ is defined in \eqref{eq:ERS+ECS} and can be computed, in our case, by the formula
$$
r_l = f^{(l)}_{i_{l+1}i_l} + \sigma^{(l)}.
$$
Then
\begin{equation}\label{eq:meas_ext_proof_Thm4.1}
\ba
\wh{\mu}_{\ov i}(\wh X_{B(\ov i)}) = &\ 1 + \sum_{n \geq 0} \frac{r_0 \ \cdots \ r_{n-1}}{f^{(0)}_{i_1i_0} \ \cdots \ f^{(n)}_{i_{n+1}i_{n}}} \sum_{j \in V_n \setminus \{i_n\}} f_{i_{n+1},j}^{(n)}\\
= &\ 1 + \sum_{n \geq 0} \frac{r_0\ \cdots \ r_{n-1}}{f^{(0)}_{i_1i_0} \ \cdots \ f^{(n)}_{i_{n+1}i_{n}}} (r_n - f^{(n)}_{i_{n+1},i_n})\\
= &\ 1 + \sum_{n \geq 0} \left(\prod_{l = 0}^{n} \frac{r_l}{f^{(l)}_{i_{l+1}i_{l}}} - \prod_{l = 0}^{n-1} \frac{r_l}{f^{(l)}_{i_{l+1}i_{l}}}\right).
\ea 
\end{equation}
Denote 
$$
\alpha_n = \prod_{l = 0}^{n} \frac{r_l}{f^{(l)}_{i_{l+1}i_{l}}}.
$$
Therefore, we have
$$
\wh{\mu}_{\ov i}(\wh X_{B(\ov i)}) < \infty \ \ \Longleftrightarrow \ \  \lim_{n \rightarrow \infty} \alpha_n < \infty \ \  \Longleftrightarrow \ \   \prod_{l = 0}^{\infty} \frac{r_l}{f^{(l)}_{i_{l+1}i_{l}}} < \infty.
$$
Note that 
$$
\frac{r_l}{f^{(l)}_{i_{l+1}i_{l}}} = 1 + \frac{\sigma^{(l)}}{f^{(l)}_{i_{l+1}i_{l}}}.
$$
Hence
$$
\wh{\mu}_{\ov i}(\wh X_{B(\ov i)}) < \infty \ \ \Longleftrightarrow \ \  \sum_{i = 0}^{\infty} \frac{\sigma^{(l)}}{f^{(l)}_{i_{l+1}i_{l}}} < \infty.
$$

\end{proof}

We will call two odometers $B(\ov i)$ and $B(\ov i')$ \textit{tail parallel} if there exists $k \in \mathbb{Z}$ such that for all sufficiently large $n$ we have $i_n' = i_n + k$. Denote by $TP(\ov i)$ the family of all odometers $B(\ov i')$ which are tail parallel to $B(\ov i)$. It is evident that two families $TP(\ov i)$ and $TP(\ov i')$ are either identical or disjoint.

\begin{prop}\label{prop:family_odom}
Let $B$ be as above and assume that the measure extension $\wh{\mu}_{\ov i}$ is finite for some odometer $B(\ov i)$. Then the measure extension $\wh{\mu}_{\ov i'}$ is finite if and only if $B(\ov i')$ is tail parallel to $B(\ov i)$.
\end{prop}

\begin{proof}
If $B(\ov i')$ is tail parallel to $B(\ov i)$ then the finiteness of  the measure extension $\wh{\mu}_{\ov i'}$ follows from the fact that $B$ is horizontally stationary. 

Assume that $B(\ov i')$ is not tail parallel to $B(\ov i)$. Denote $k_n = i'_n - i_n$. Then there is an increasing sequence $\{n(j)\}_{j=0}^{\infty}$ such that $k_{n(j)} \neq k_{n(j)+1}$. For such $n(j)$ we have $i'_{n(j)} = i_{n(j)} + k_{n(j)} \neq i_{n(j)} + k_{n(j)+1}$. Since $\wh{\mu}_{\ov i}$ is finite, by Theorem~\ref{Thm:ext_from_odom} we have
$$
\sigma^{(n)} = \sum_{j\neq i_n}f^{(n)}_{i_{n+1},j} < f^{(n)}_{i_{n+1}i_n}
$$
for all $n$ greater than some $N$. For $n = n(j)$ large enough, we also have 
$$
f^{(n)}_{i'_{n+1}i'_n} \leq \sum_{j\neq k_{n+1}+i_n}f^{(n)}_{i'_{n+1},j} = \sum_{j\neq i_n}f^{(n)}_{i_{n+1},j} < f^{(n)}_{i_{n+1},i_n} = f^{(n)}_{i'_{n+1},i_n + k_{n+1}} \leq \sum_{j\neq 
 i'_n}f^{(n)}_{i'_{n+1},j} = \sigma'^{(n)}.
$$
Hence for infinitely many $n$, we have
$$
\frac{\sigma'^{(n)}}{f^{(n)}_{i'_{n+1}i'_n}} \geq 1,
$$
and, by Theorem~\ref{Thm:ext_from_odom}, the extension 
$\wh{\mu}_{\ov i'}$ is infinite.
\end{proof}

\begin{remark}
 (1) Let $B(\ov i)$ be an odometer such that the measure extension $\wh{\mu}_{\ov i}$ is finite. Then it follows from Proposition~\ref{prop:family_odom} that for every $k \in \mathbb{Z}$, the measures $\wh{\mu}_{{\ov i} + k}$ are also finite. These measures are pairwise singular and form (after normalization) a countable family of probability ergodic invariant measures on $B$. If $B(\ov i)$ is an odometer such that $i'_n = i_n$ for $n$ large enough then, after normalization, $\wh{\mu}_{\ov i'} = \wh{\mu}_{\ov i}$.

 (2) Let $f^{(n)}_{ij_{(max,n,i)}} = \max\{f^{(n)}_{ij} : j \in V_n\}$ be the maximal element in each row $i$ of $F_n$  for  $n\in \mathbb{N}_0$ and $i \in V_{n+1}$.
 We call an odometer $B(\ov i)$ \textit{dominating} if $i_n = j_{(max,n,i)}$ for all $n$. It is easy to see that such an odometer always exists for a horizontally stationary Bratteli diagram. It follows from Theorem~\ref{Thm:ext_from_odom}, that if an odometer $B(\ov i)$ has finite measure extension $\wh{\mu}_{\ov i}$ then $B(\ov i)$ is dominating. Thus, to check if there exists an odometer $B(\ov i)$ with finite measure extension $\wh{\mu}_{\ov i}$ it is enough to check condition~\eqref{eq4:ext from odom} only for dominating odometers. Moreover, if there exist two dominating odometers $B(\ov i)$ and $B(\ov i')$ with disjoint families $TP(\ov i)$ and $TP(\ov i')$ then there is no odometer in $B$ with finite measure extension. Therefore, if there exists a dominating odometer $B(\ov i)$ satisfying $\sum_{n = 0}^{\infty} \frac{\sigma^{(n)}}{f^{(n)}_{i_{n+1}i_{n}}} < \infty$, then 
the set $TP(\ov i)$ coincides with the set of all odometers
with finite measure extension.
 
\end{remark}

Analyzing the proof of Theorem \ref{Thm:ext_from_odom}, we note that
the proved result can be extended to a wider class of generalized Bratteli diagrams.

\begin{corol}\label{cor:finite_meas_ext}
Let $B = B(F_n)$ be a generalized Bratteli diagram such that
every incidence matrix $F_n$ belongs to the class $ERS(r_n)$. Let $B(\ov i)$, $\mu_{\ov i}$ and $\sigma^{(n)}$ be as in Theorem~\ref{Thm:ext_from_odom}. Then 
\begin{equation*}
\wh{\mu}_{\ov i}(\wh X_{B(\ov i)}) < \infty \ \ \Longleftrightarrow \ \ 
\sum_{n = 0}^{\infty} \frac{\sigma^{(n)}}{f^{(n)}_{i_{n+1}i_{n}}} < \infty. 
\end{equation*}

\end{corol} 

This corollary can be proved exactly as Theorem \ref{Thm:ext_from_odom}. 

\medskip
It is natural to ask the following question. 

\noindent \textbf{\textit{Question.}} 
Let $B(F_n)$, $\sigma^{(n)}$, and $\mu_{\ov i}$ be as in Theorem 
\ref{Thm:ext_from_odom} and 
$$
\sum_{n = 0}^{\infty} \frac{\sigma^{(n)}}{f^{(n)}_{i_{n+1}i_n}} < \infty. 
$$
Does the set of all probability ergodic invariant measures on $B$ coincide (after normalization) with the set $\{\wh{ \mu}_{\ov i}: 
i \in \mathbb{Z}\}$?
\medskip

We answer this question for a class $\mathcal C$ of horizontally stationary Bratteli diagrams.  This class consists of the generalized 
Bratteli diagrams $B = B(F_n)$  whose incidence matrices  $F_n = (f_{ij}^{(n)})$ are three-diagonal and satisfy the two properties:
\be\label{eq4: 3-diag inc matr}
f_{ij}^{(n)} = \begin{cases}
a_n \geq 1, & j = i\\
1, & j= i+1 \mbox{ or } j = i-1\\
0, & \mbox{otherwise},
\end{cases}
\ee
and 
\begin{equation}\label{eq:caseIseriesconv}
\sum_{n = 0}^{\infty} \frac{1}{a_n} < \infty.
\end{equation}

\begin{theorem}\label{thm all erg meas}
If a horizontally stationary Bratteli diagram $B$ belongs to the class $\mathcal C$, then the set of all probability ergodic tail invariant measures on $B$ coincides with the set $\{\wh{ \mu}_{\ol i}
: i \in \Z\}$ (after normalization), where the measure 
$\wh{\mu_{\ol i}}$ is the extension of $\mu_{\ol i}$, 
$\ov i = (i,i,i,\ldots)$, the unique probability tail invariant ergodic measures supported by the $i$-th odometer. 
\end{theorem}

\begin{proof}
For such a diagram $B$, we easily find that 
$$
H_i^{(n)} = (a_0 + 2) \ \cdots \ (a_{n-1} + 2), \ \ n \geq 1,
$$
for all $i \in \mathbb{Z}$.
Let 
$$
G'^{(n,m)} = F_{n+m-1}\ \cdots\ F_{n}, 
$$
that is $G'^{(n,m)} = (g'^{(n,m)}_{ij} : i \in V_{n+m}, j \in V_n)$
where $g'^{(n,m)}_{ij}$ indicates the number of paths between 
the vertices $i \in V_{n+m}$ and $j \in V_n$.
It follows from the definition of $F_n$ that 
$$
g'^{(n,m)}_{ij} = 0 \; \mbox{ whenever } \; |i - j| > m.
$$
One can prove the following property of entries $g'^{(n,m)}_{ij}$ using the induction with respect to $m$.

\begin{claim}\label{claim:behavior_g}
For every fixed $i$, the sequence $\{g'^{(n,m)}_{ij}\}$ is increasing if $j = i - m, \ldots, i - 1$ and decreasing if $j = i, i+1, \ldots, i + m - 1$. Moreover,
    $$
    \max\{g'^{(n,m)}_{ij} \; | \; j = i - m, \ldots, i + m\} = g'^{(n,m)}_{ii}.
    $$
\end{claim}

To compute $g'^{(n,m)}_{ii}$ explicitly, we observe the following fact. 
For this Bratteli diagram, to pass from the vertex $i \in V_{n+m}$ to the vertex $i \in V_n$, we can either go through $m$ ``vertical'' edges, or we can choose to go through $k \geq 1$ edges 
 slanted from left to right, and through the same amount, $k$, of edges slanted from right to left to end up in the vertex $i$ on level $V_n$. Thus, we obtain the following formula:
 $$
\begin{aligned}
g'^{(n,m)}_{ii} =&\  a_n\ \cdots\ a_{n + m - 1} + \sum_{2 \leq 2k \leq m} \binom{2k}{k} \left[\sum_{S \subset \{n, \ldots, n + m - 1\}} \ \ \prod_{j \in \{n, \ldots, n + m -1\} \setminus S} a_j\right]\\
=&\ \sum_{0 \leq 2k \leq m} \binom{2k}{k} \left[\sum_{S \subset \{n, \ldots, n + m - 1\}} \ \ \prod_{j \in \{n, \ldots, n + m -1\} \setminus S} a_j\right],
\end{aligned}
$$
where $S$ is a subset of $\{n, \ldots, n + m - 1\}$ 
of cardinality $2k$. 

Further, the entries of the stochastic matrices $G^{(n,m)}$ 
are determined by the formula
\begin{equation}\label{eq:formula_gij}
g_{ij}^{(n,m)} = \frac{g'^{(n,m)}_{ij}}{(a_n + 2)\ \cdots\ 
(a_{n + m - 1}+2)} \; \mbox{ for } \; |i - j| \leq m.
\end{equation}

Note that, by Theorem \ref{Thm:ext_from_odom}, inequality \eqref{eq:caseIseriesconv} implies that every measure $\wh{\mu}_{\ov i}$ is finite. Indeed, for every $\ov i = (i,i,i,\ldots)$ we have $\sigma^{(n)} = 2$ and $f^{(n)}_{i_{n+1}i_n} = a_n$ for $n = 0,1,2 \ldots$
Thus, the family $\{\wh{\mu}_{\ov i} : i \in \mathbb{Z}\}$ is (after normalization) a family of ergodic tail invariant probability measures on $B$ (see \cite{BezuglyiKarpelKwiatkowski2024}). To prove that they form the set of all ergodic invariant probability measures, we use the inverse limit method developed in \cite{BezuglyiKarpelKwiatkowskiWata} and consider the set of probability vectors $\ov g^{(n,m)}_{i} = (g^{(n,m)}_{ij})_{j \in \Z}$ and find all limit points $\ov x^{(n)} = \lim_{m \rightarrow \infty}\ov g^{(n,m)}_{i_m}$ for every fixed number $n$. The inverse limit method allows to identify all ergodic probability tail invariant measures on the path space of a generalized Bratteli diagram with inverse limits of infinite-dimensional simplices associated with levels of the diagram. This method is a generalization of a similar approach 
developed for standard Bratteli diagrams in \cite{BezuglyiKwiatkowskiMedynetsSolomyak2010}, \cite{AdamskaBKK2017}, \cite{BezuglyiKarpelKwiatkowski2019}. The case of generalized Bratteli diagrams is a lot more complicated, since one has to deal with infinite-dimensional simplices instead of the finite-dimensional ones. 
For a classical Bratteli diagram with the sequence of incidence stochastic matrices $(F_n)$, every tail invariant measure $\mu$ is completely 
determined by a sequence of non-negative probability 
vectors $(\ol q^{(n)})$ such
that $F_n^T \ol q^{(n+1)} = \ol q^{(n)}$ for all $n \geq 1$, and the coordinates of vector $\ol q^{(n)}$ correspond to the measure of the towers on level $n$
(see e.g. \cite[Equation (2.2) and Theorem 2.5]{BezuglyiKarpelKwiatkowski2019} and Theorem \ref{thm:mu_as_inv_lim}). 
In other words, the set $M_1(\mc R)$ of all probability tail invariant 
measures on $X_B$ can be identified with the inverse limit
of the sets $(\De_1^{(n)}, F_n^T)$:
$$
M_1(\mc R) = \varprojlim_{n\to \infty} (\De_1^{(n)}, F_n^T),
$$
where $\De_1^{(n)}$ is the finite-dimensional simplex 
indexed by the vertices of the $n$-th level. When working with generalized Bratteli diagrams, one encounters many technical difficulties, one of them is that the infinite-dimensional simplex $\De_1^{(n)}$ is not closed. In \cite[p. 25]{BezuglyiKarpelKwiatkowskiWata}, is presented an algorithm for finding the set of all ergodic probability tail invariant measures for a generalized Bratteli diagram.

\begin{claim}\label{rem:zero_lim_point}
If for some $n$ the set of all limit points $\{\ov x^{(n)}\}$ consists only of a zero vector then there is no invariant probability measure on $B$ \cite{BezuglyiKarpelKwiatkowskiWata}.
\end{claim} 

We first prove that if $\{i_m\}_{m \in \N}$ is unbounded, then for every $n$, every limit point $\ov x^{(n)}$ is a zero vector. It follows from Claim \ref{claim:behavior_g} that for every $\varepsilon > 0$, there is a natural number $p$ such that, for every $m$, the inequality $g^{(n,m)}_{ij} < \varepsilon$ holds whenever 
$|i - j| > p$ (we can pick any natural number $p \geq \frac{1}{2\varepsilon}$). Indeed, assume that the contrary holds. Then there exists $\varepsilon_0 > 0$ such that for every $p$ we can find $m, i, j$
with $|i-j| > p$ and $g^{(n,m)}_{ij} \geq \varepsilon_0$. Then by Claim \ref{claim:behavior_g}, all elements $g^{(n,m)}_{ij}$ for $j \in \{i - p, \ldots, i + p\}$ are greater then $\varepsilon_0$, and we obtain
$$
1 = \sum_{j \in \mathbb{Z}} g^{(n,m)}_{ij} \geq (2p+1)\varepsilon_0. 
$$
Taking $p$ large enough, we get a contradiction. 

Fix $j \in \mathbb{Z}$ and find infinitely many $i_m$ such that $|i_m  - j| > p$. Then 
$$
x_j^{(n)} = \lim_{m \rightarrow \infty} g^{(n,m)}_{i_mj} \leq \varepsilon
$$
for any $\varepsilon > 0$ which implies that $x_j^{(n)} = 0$ and $\ov x = (x_j^{(n)})$ is a zero vector.

Thus, without loss of generality, we can assume that the sequence 
$\{i_m\}_{m \in \N}$ is bounded. 
Passing to subsequences, if needed, we can state that every non-zero limit $\ov x^{(n)}$ of the vectors $\ov g^{(n,m)}_{i_m}$ has the form $\ov x^{(n)} = \lim_{m \rightarrow \infty} \ov g^{(n,m)}_{i}$ for some fixed 
$i \in \mathbb{Z}$. 

Let $\nu_i$ be an ergodic probability tail invariant measure defined by vectors $\ov x^{(n)}$ (we refer to \cite{BezuglyiKarpelKwiatkowskiWata}). Then, for every cylinder set $[\ov e]$ with $r(\ov e) = j \in V_n$, we have
$$
H_j^{(n)}\nu_i([\ov e]) = (a_0 + 2)\ \cdots \ (a_{n-1} + 2)
\nu_i([\ov e]) = \lim_{m \rightarrow \infty} \frac{g'^{(n,m)}_{ij}}{(a_n + 2)\ \cdots \ (a_{n + m - 1} + 2)}.
$$
On the other hand, we can compute $\wh \mu_{\ov i}([\ov e])$ for $\ov i = (i,i,i \ldots)$ using formula (3.13) from \cite{BezuglyiKarpelKwiatkowskiWata}:
$$
\widehat{{\mu}_{\ov i}}([\ol e]) =
\lim_{m \to \infty} \left\lbrack \sum\limits_{i \in  
V_{n + m}}^{}{{g'}_{ij}^{(n,m)}} 
{\mu_{\ov i}( [\ov e^{(n + m)}_i]) }\right\rbrack = \lim_{m \to \infty} \frac{g'^{(n,m)}_{ij}}{a_0\ \cdots \ a_{n + m - 1}}
,\quad 
j = r(\ol e),
$$
where $\ov e^{(n + m)}_i$ is a finite path in $B(\ov i)$ which ends at the vertex $i$ on  level $n + m$.
Condition \eqref{eq:caseIseriesconv} implies that
$$
\alpha = \prod_{n = 0}^{\infty} \frac{a_n + 2}{a_n} < \infty
$$
which, in its turn, implies that $\wh \mu_{\ov i}$ is equivalent to the probability ergodic invariant measure $\nu_i$. 
Indeed, we apply de Possel's theorem (see, for instance, \cite{ShilovGurevich1977}) for the cylinder sets and get that 
$$
\frac{\nu_i([\ov e])}{\widehat{{\mu}_{\ov i}}([\ol e])} = \alpha
$$
(see also (2.17) in \cite{BezuglyiJorgensen2024}).
Thus, the family of measures $\{\wh{ \mu}_{\ov i}: 
i \in \mathbb{Z}\}$ coincides (after normalization) with the set of all ergodic invariant probability measures for $B$.
\end{proof}

\begin{remark}
It follows from \eqref{eq:formula_gij} and Claim \ref{claim:behavior_g} that 
\begin{equation}\label{eq:zero_gi}
\lim_{m \rightarrow \infty} \ov g^{(n,m)}_{i} = 0 \Longleftrightarrow \lim_{m \rightarrow \infty} g^{(n,m)}_{ii} = 0, \mbox{ where } i = i_m.
\end{equation}
Note that since $B$ is horizontally stationary, the value of $g^{(n,m)}_{ii}$ does not depend on $i$.
It is important to have explicit formulas for the computation of 
entries $g_{ii}^{(n,m)}$. We can do it as follows.

Denote by $\Lambda = \{n, \ldots, n+m-1\}$. Then we write 
\begin{equation}\label{eq:formula_gii}
\begin{aligned}
g_{ii}^{(n,m)} = &\ \frac{1}{(a_n + 2)\ \cdots \ (a_{n+m-1}+2)}\left(\sum_{0 \leq 2k \leq m} \binom{2k}{k} \left[\sum_{S \subset \Lambda} \ \ \prod_{j \in \Lambda  \setminus S} a_j \right]\right)\\
= &\ \sum_{0 \leq 2k \leq m} \binom{2k}{k} \left[\sum_{S \subset \Lambda} \prod_{j \in S}\frac{1}{a_j + 2} \prod_{j \in \Lambda \setminus S} \frac{a_j}{a_j + 2}\right]\\
= &\ \left(\prod_{j = n}^{n + m - 1} \frac{a_j}{a_j + 2}\right) \left(\sum_{0 \leq 2k \leq m} \binom{2k}{k} \left[\sum_{S \subset 
\Lambda} \ \ \prod_{j \in S} \frac{1}{a_j}\right] \right),
\end{aligned} 
\end{equation}
where $|S| = 2k$.
\end{remark}

Let $B = B(F_n)$ be a horizontally stationary Bratteli diagram with the three-diagonal incidence matrices $F_n$ (considered in Theorem
\ref{thm all erg meas}) such that the condition \eqref{eq:caseIseriesconv} is not satisfied, i.e.
$$
\sum_{n = 0}^{\infty} \frac{1}{a_n} = \infty.
$$
Now we can prove the following result.

\begin{prop}\label{prop:no_inv_meas} Let $B = B(F_n)$ be a 
a generalized Bratteli diagram with three-diagonal incidence matrices $F_n$ satisfying \eqref{eq4: 3-diag inc matr}. Suppose that 
$\sum_n a_n^{-1} = \infty$. Then 
there is no tail invariant probability measure on $X_B$ 
    if and only if there is $n \in \N_0$ such that, for all $l \in \mathbb{N}$, we have
    \begin{equation}\label{eq:no_measure}
    \lim_{m \rightarrow \infty} \left(\prod_{j = n}^{n + m - 1} \frac{a_j}{a_j + 2}\right) \left(\sum_{\substack{S \subset 
\Lambda, |S| = l}} \prod_{j \in S}\frac{1}{a_j}\right) = 0
    \end{equation}
(the expression in second brackets makes sense only for $m$ large enough such that $m \geq l$).    
 \end{prop}

 \begin{remark}
     Note that in equation \eqref{eq:no_measure}, the set $S$ can have arbitrary size, even or odd. This property follows from the equation \eqref{eq:formula_a0} and estimates \eqref{eq:formula_a}. 
 \end{remark}

\begin{proof}
By \eqref{eq:zero_gi} and Claim \ref{rem:zero_lim_point}, it is enough to prove that $\lim_{m \rightarrow \infty} g^{(n,m)}_{ii} = 0$ if and only if \eqref{eq:no_measure} holds for all $l \in \mathbb{N}$. By \eqref{eq:formula_gii} and since $\sum_n a_n^{-1} = \infty$, it is clear that $\lim_{m \rightarrow \infty} g^{(n,m)}_{ii} = 0$ implies that \eqref{eq:no_measure} holds for all $l \in \mathbb{N}$.

Conversely,  prove that \eqref{eq:no_measure}  implies that $\lim_{m \rightarrow \infty} g^{(n,m)}_{ii} = 0$ for all $l$.
Take any $\varepsilon > 0$ and pick a natural number $k_{\varepsilon}$ such that
    $$
    \frac{1}{2^{k}}\binom{2k}{k} < \varepsilon
    $$
    for every $k > k_{\varepsilon}$.
    
It follows from 
\eqref{eq:formula_gii} that for all $m \geq k_{\varepsilon}$ 
    $$
    \ba
    g_{ii}^{(n,m)} = &\
    \left(\prod_{j = n}^{n + m - 1} \frac{a_j}{a_j + 2}\right) \left( 1 + \sum_{2 \leq 2k \leq m} \binom{2k}{k} \left[\sum_{\substack{S \subset \Lambda, \ |S| = 2k}}\ \ \prod_{j \in S} \frac{1}{a_j}\right] \right)\\
    = &\ \left(\prod_{j = n}^{n + m - 1} \frac{a_j}{a_j + 2}\right) \left( 1 + \sum_{2 \leq 2k \leq 2k_{\varepsilon}} \binom{2k}{k} \left[\sum_{\substack{S \subset \Lambda,\ |S| = 2k}}\ \  \prod_{j \in S} \frac{1}{a_j}\right]\right.\\
 & \ \ \   + \ \left. \sum_{2k_{\varepsilon} < 2k \leq m}  \frac{1}{2^{k}}\binom{2k}{k} 2^k \left[\sum_{\substack{S \subset \Lambda, \ |S| = 2k}}\ \  \prod_{j \in S} \frac{1}{a_j}\right]\right)\\
    \leq &\ \left(\prod_{j = n}^{n + m - 1} \frac{a_j}{a_j + 2}\right) \left( 1 + \binom{2k_{\varepsilon}}{k_{\varepsilon}}\sum_{2 \leq 2k \leq 2k_{\varepsilon}}  \left[\sum_{\substack{S \subset \Lambda,\ |S| = 2k}}\ \  \prod_{j \in S} \frac{1}{a_j}\right]\right.\\
 & \ \ \  + \ \left. \varepsilon \sum_{2k_{\varepsilon} < 2k \leq m} 2^k \left[\sum_{\substack{S \subset \Lambda, \ |S| = 2k}}\ \  \prod_{j \in S} \frac{1}{a_j}\right]\right).\\
  \ea
    $$

Note that
\begin{equation}\label{eq:formula_a0}
\ba
(a_n + 2)\ \cdots\ (a_{n+m-1}+2) = &\  a_n\ \cdots\ a_{n + m - 1} + \sum_{1 \leq k \leq m} 2^k \left(\sum_{\substack{S \subset \Lambda,  
 |S| = k}} \left[\prod_{j \in \Lambda  \setminus S}a_j\right]\right)\\
= &\ a_n\ \cdots \ a_{n + m - 1} \left[ 1 + \sum_{1 \leq k \leq m} 2^k \left(\sum_{\substack{S \subset \Lambda, |S| = k}} \left[\prod_{j \in S}\frac{1}{a_j}\right]\right)\right].
\ea
\end{equation}    
From \eqref{eq:formula_a0}, we obtain
\begin{equation}\label{eq:formula_a}
\ba
1 = &\  \left(\prod_{j = n}^{n + m - 1}\frac{a_j}{a_j + 2}\right) \left[ 1 + \sum_{1 \leq k \leq m} 2^k \left(\sum_{\substack{S \subset \Lambda, |S| = k}} \left[\prod_{j \in S}\frac{1}{a_j}\right]\right)\right]\\
\geq &\ \left(\prod_{j = n}^{n + m - 1}\frac{a_j}{a_j + 2}\right) \sum_{1 \leq k \leq m} 2^k \left(\sum_{\substack{S \subset \Lambda, |S| = k}} \left[\prod_{j \in S}\frac{1}{a_j}\right]\right)\\
\geq &\ \left(\prod_{j = n}^{n + m - 1}\frac{a_j}{a_j + 2}\right)\sum_{2k_{\varepsilon} \leq 2k \leq m} 2^k \left(\sum_{\substack{S \subset \Lambda, |S| = 2k}} \left[\prod_{j \in S}\frac{1}{a_j}\right]\right).
\ea
\end{equation}

Thus, we get
$$
  g_{ii}^{(n,m)} \leq \left(\prod_{j = n}^{n + m - 1} \frac{a_j}{a_j + 2}\right) \left( 1 + \binom{2k_{\varepsilon}}{k_{\varepsilon}}\sum_{2 \leq 2k \leq 2k_{\varepsilon}}  \left[\sum_{\substack{S \subset \Lambda, \ |S| = 2k}} \ \ 
    \prod_{j \in S} \frac{1}{a_j}\right] \right) +  \varepsilon.
 $$   
Since $\sum_n a_n^{-1} = \infty$ and by \eqref{eq:no_measure}, we can find $m_{\varepsilon}$ such that for all $m > m_{\varepsilon}$ we have
    $$
    \prod_{j = n}^{n + m - 1} \frac{a_j}{a_j + 2} < \varepsilon
    $$
    and
    $$
    \prod_{j = n}^{n + m - 1} \frac{a_j}{a_j + 2} \cdot \binom{2k_{\varepsilon}} {k_{\varepsilon}} \cdot \left(\sum_{2 \leq 2k \leq 2k_{\varepsilon}} \ \ \sum_{\substack{S \subset \Lambda, \ |S| = 2k}}\ \ \prod_{j \in S} \frac{1}{a_j}\right) < \varepsilon.
    $$
Hence, we have proved that $g_{ii}^{(n,m)} < 3\varepsilon$ whenever $m > m_{\varepsilon}$, i.e.,
    $$
    \lim_{m \rightarrow \infty} g^{(n,m)}_{ii} = 0.
    $$
    It follows that the unique limit point of any sequence of vectors $\ov g^{(n,m)}_i$ is zero. This means there is no invariant probability tail invariant measure on $B$.
\end{proof}

\begin{example} Suppose that a Bratteli diagram $B$ from the class $\mathcal C$ is defined by the incidence matrices $F_n$ where 
 $a_n = a$ for all $n \in \mathbb{N}_0$,  $a \in \mathbb{N}$. Then
    $$
    \lim_{m \rightarrow \infty} \prod_{j = n}^{n + m - 1} \frac{a}{a+2}= \lim_{m \rightarrow \infty} \left(\frac{a}{a+2}\right)^m = 0
    $$
    for every $n = 0,1,\ldots$ Moreover, we have
    $$
    \sum_{\substack{S \subset \Lambda,\ |S| = k}}\ \  \prod_{j \in S}\frac{1}{a_j} = \frac{1}{a^k}\binom{m}{k}
    $$
    and
    $$
     \lim_{m \rightarrow \infty} \left(\prod_{j = n}^{n + m - 1} \frac{a_j}{a_j + 2}\right) \left(\sum_{\substack{S \subset \Lambda, \ |S| = k}} \ \ \prod_{j \in S}\frac{1}{a_j}\right) = \lim_{m \rightarrow \infty} \left(\frac{a}{a+2}\right)^m \frac{1}{a^k}\binom{m}{k} = 0.
    $$
    By Proposition \ref{prop:no_inv_meas}, there is no invariant tail invariant probability measure on $B$.
\end{example}

\begin{example} Assume now that the main diagonal entries $F_n$ are 
     $a_n = n + 1$ for all $n \in \mathbb{N}_0$. Then 
    $$
    \prod_{j = n}^{n + m - 1} \frac{a_j}{a_j + 2} = \frac{n(n+1)}{(n+m)(n+m+1)}.
    $$
    We also have
    $$
    \sum_{\substack{S \subset \Lambda\\ |S| = k}} \prod_{j \in S}\frac{1}{a_j} \leq \left( \frac{1}{n} + \ldots + \frac{1}{n+m}\right)^k \leq \ln^k(n+m).
    $$
    Thus, we get 
    $$
    \lim_{m \rightarrow \infty} \left(\prod_{j = n}^{n + m - 1} \frac{a_j}{a_j + 2}\right) \left(\sum_{\substack{S \subset \Lambda, \  |S| = k}} \ \ \prod_{j \in S}\frac{1}{a_j}\right) \leq \lim_{m \rightarrow \infty}\frac{n(n+1)}{(n+m)(n+m+1)}\ln^k(n+m) = 0.
    $$
    Therefore, there is no tail invariant probability measure on $B$.
\end{example}

\subsection{Measure extension from ECS subdiagrams}

Let $B$ be a horizontally stationary generalized Bratteli diagram and let $\ov B$ be a vertex subdiagram of $B$ defined by the sequence of vertices 
$(W_n)$ where 
$W_n \subset V_n$ and  $|W_n| < \infty$. Assume that the incidence 
matrices $\ol F_n$ of $\ol B$  have the ECS property. Then $\ov B$ admits a probability tail invariant measure $\ov\mu$ such that, for every $v \in W_{n+1}$, one has
$$
\ov \mu([\ov e^{(n+1)}(v)]) = \frac{1}{c_0 \ \cdots \ c_n},
$$
where $c_n = \sum_{i \in W_{n+1}} \ol f_{ij}^{(n)}$ for every $j \in W_{n}$. Then the following generalization of Theorem~\ref{Thm:ext_from_odom} holds:

\begin{theorem}\label{Thm:ext_from_ECS}
Let $B(F_n)$, $\ov B$ and $\ol\mu$ be as above.
Then 
$$
\wh{\ov \mu}(\wh X_{\ov B}) < \infty \ \ \Longleftrightarrow \ \ 
\prod_{i = 0}^{\infty} \frac{r_i}{c_i} < \infty \ 
\mbox{and}\ \sup \{|W_n| : n \in \N_0\} < \infty.
$$
\end{theorem}

\begin{proof}
We first compute the measures extension $\wh{\ov \mu}(\wh X_{\ov B})$ 
using the method from \cite{AdamskaBKK2017},
\cite{BezuglyiKarpelKwiatkowskiWata}:
$$
\ba
\wh{\ov \mu}(\wh X_{\ov B}) = &\ 1 + \sum_{n \geq 0} \sum_{i \in W_{n+1}} \sum_{j \in W_n'} f_{ij}^{(n)} \ov p_i^{(n+1)} H_j^{(n)}\\
 = &\ 1 + \sum_{n \geq 0} \frac{r_0 \ \cdots \ r_{n-1}}{c_0 \ \cdots \ c_{n}}\sum_{i \in W_{n+1}} \sum_{j \in W_n'}  f_{ij}^{(n)},\\
\ea 
$$
where $W'_n = V_n \setminus W_n$ and $\ov p_i^{(n+1)} = \ov \mu([\ov e^{(n+1)}(i)])$ is the measure $\ov \mu$ of a cylinder set which ends at the vertex $i \in V_{n+1}$.
We have
$$
r_n = \sum_{j \in V_{n}}f^{(n)}_{ij} = \sum_{j \in W_{n}}f^{(n)}_{ij} + \sum_{j \in W'_{n}}f^{(n)}_{ij}.
$$
Denote 
$$
\ov r_i^{(n)} = \sum_{j \in W_{n}}f^{(n)}_{ij}.
$$
Then
$$
\ba
\wh{\ov \mu}(\wh X_{\ov B}) = &\ 1 + \sum_{n \geq 0} \frac{r_0 \ \cdots \ r_{n-1}}{c_0 \ \cdots \ c_{n}}\sum_{i \in W_{n+1}} \left(r_n - \ov r_i^{(n)}\right)\\
= &\ 1 + \sum_{n \geq 0} \frac{r_0 \ \cdots \ r_{n-1}}{c_0 \ \cdots \ c_{n}}\left(r_n|W_{n+1}| - c_n |W_n|\right)\\
= &\ 1 + \sum_{n \geq 0} \left(\prod_{i = 0}^{n} \frac{r_i}{c_i} |W_{n+1}| - \prod_{i = 0}^{n-1} \frac{r_i}{c_i}|W_n|\right).
\ea 
$$
Let 
$$
\alpha_n = \prod_{i = 0}^{n} \frac{r_i}{c_i}|W_{n+1}|.
$$
Then we use the fact that $r_i \geq c_i$ for all $i$ to deduce the following implications:
$$
\wh{\ov \mu}(\wh X_{\ov B}) < \infty \ \ \Longleftrightarrow \ \  \lim_{n \rightarrow \infty} \alpha_n < \infty \ \  \Longleftrightarrow \ \   \prod_{i = 0}^{\infty} \frac{r_i}{c_i} < 
\infty \ \mbox{and} \ \sup_n\{ |W_{n}|\}  < \infty.
$$

\end{proof}

The following example illustrates Theorem \ref{Thm:ext_from_ECS}.

\begin{example}
Let  $B=(V, E)$ be the horizontally stationary Bratteli diagram defined by the sequence of matrices $(F_n)$ such that 
$f_{ij}^{(n)}= 1$ if $j=i-1$  or  $j=i+1$,   
$f_{ii}^{(n)}= a_n$,   and  $f_{ij}^{(n)}=0$ when $|j -i| > 1$ ,
$i \in \Z$. Take the subdiagram $\ol B$ such   
$W_n   = \{0,1\}$  for all $n$. Then the incidence matrices $F'_n$
of the vertex subdiagram   $\ol B=(\ol V, \ol E)$ have the form 
$$
F'_n = \begin{pmatrix}
    a_n & 1\\
    1 & a_n
\end{pmatrix}.
$$
Thus,  we have  $r_n = a_n +2$  and  $c_n = a_n +1$, $n \in \N_0$.
Let  $\ol \mu$  be a tail invariant measure on the path space 
$X_{\ol B}$ of $\ol B$ defined as in Theorem \ref{Thm:ext_from_ECS}. 
It follows from this theorem that the extended measure $\wh{\ol \mu}$
is finite if and only if 
$$
\prod_{n = 0}^\infty \frac{r_n}{c_n} < \infty 
\ \ \Longleftrightarrow \ \   \sum_{n=0}^\infty \frac{1}{a_n} < \infty.
$$

The subdiagram  $\ol B$ considered above contains two vertical 
odometers, $\ol B_0$ and $\ol B_1$, determined by the vertices 
$\{0\}$ and $\{1\}$, respectively. Let  $\mu_0$ and $\mu_1$ 
be the unique tail invariant probability ergodic measures on 
the path spaces of $\ol B_0$ and $\ol B_1$, respectively. The condition $\sum_{n=0}^\infty a_n^{-1} < \infty$
implies that the extended measures $\wh \mu_0$ and $\wh \mu_1$ are
finite on $X_{\ol B}$. Since they are mutually singular, 
$\wh \mu_0$ and $\wh \mu_1$ form the set of all ergodic finite probability tail invariant measures on $X_{\ov B}$ (see also \cite{AdamskaBKK2017}). Furthermore, the measure $\ol\mu$
has the following property: $\ol\mu(X_{\ol B_i}) > 0$, $ i= 0,1$, which easily follows from the convergence of the series $\sum_{n=0}^\infty a_n^{-1}$ because 
\begin{equation}\label{eq:mu_hat_ex4.12}
\wh{\ol \mu}(X_{\ol B_0}) = \wh{\ol \mu}(X_{\ol B_1}) =
\prod_{i=0}^\infty \frac{a_i}{a_i +1}. 
\end{equation}
This means that the measure $\wh{\ol \mu}$ defined above in this example is a convex combination of the ergodic measures  $\wh \mu_0$ and $\wh \mu_1$. Moreover, from \eqref{eq:mu_hat_ex4.12} it follows that $\wh{\ov \mu}$ is not an ergodic measure.

In the case when $\sum_{n=0}^\infty a_n^{-1} = \infty$, the extended measures $\nu_0$ and $\nu_1$ on $X_{\ov B}$ are infinite, the  measure $\ov \mu$ is a unique invariant probability measure on $\ov B$. Clearly, the extensions $\wh \mu_0$ and $\wh \mu_1$ of the measures $\mu_0$, $\mu_1$ to the whole path space $X_B$ are also infinite. It can be also seen from the formula \eqref{eq:meas_ext_proof_Thm4.1}. The fact that $\ov \mu$ is the unique invariant probability measure on $\ov B$ follows from Proposition 3.1 in \cite{AdamskaBKK2017} and can be also proved using the inverse limit method.

Clearly, the same approach can be used for vertex subdiagrams $\ol B$ supported by any finite number of vertices, $W_n = \{0, 1, ..., k\}$,
$n \in \N_0$. The $(k+1)\times(k+1)$ incidence matrices $F_n'$ of $\ov B$ have the form
$$
F'_n = \begin{pmatrix}
    a_n & 1 & 0  & \ldots & 0 & 0 & 0\\
    1 & a_n & 1 & \ldots & 0 & 0 & 0\\
    0 & 1 & a_n  & \ldots & 0 & 0 & 0\\ 
  \vdots & \vdots & \ddots & \ddots & \ddots &  \vdots & \vdots \\
    0 & 0 & 0 & \ldots & a_n & 1 & 0\\     
    0 & 0 & 0 & \ldots & 1 & a_n & 1\\    
    0 & 0 & 0 & \ldots & 0 & 1 & a_n\\
\end{pmatrix}.
$$

Let us observe that the diagram $\ov B$ does not satisfy the ECS property if $k > 1$. The subdiagram $\ov B$ contains $k + 1$ ``vertical'' odometers $\ov B_i$ with invariant probability measures $\mu_i$ for $i = 0, 1, \ldots, k$. If $\sum a_n^{-1} < \infty$ then by the same arguments as above we can prove that the measures $\mu_i$ have finite extensions $\nu_i$ on $\ov B$ and $\wh \mu_i$ on $B$ for $i = 0, 1, \ldots, k$. Moreover, the measures $\{\nu_i\}_{i = 0}^k$ are (after normalization) all ergodic probability invariant measures on $\ov B$.
\end{example}

\subsection{Horizontally invariant measures} 
In this subsection, we study Markov measures on the path space $X_B$
of a generalized Bratteli diagrams. Such measures were considered 
in \cite{DooleyHamachi2003}, \cite{Renault2018}, 
\cite{BezuglyiJorgensen2022}, and some other papers. 
The fact that a generalized Bratteli diagram is horizontally stationary 
allows us to work with Markov measures satisfying 
 the invariance property with respect to horizontal translations. 

We begin with the definition of a Markov measure. 

\begin{definition}\label{def Mark meas} 
Let $B = (V, E)$ be  a generalized Bratteli diagram
constructed by 
a sequence of incidence matrices $(F_n)$.  Let $q = (q_{v})$ be a 
strictly positive  vector, $q_v >0, v\in V_0$, and let $(P_n)$ 
be a sequence of non-negative infinite matrices with entries 
$(p^{(n)}_{v,e})$ where $v \in V_n, e \in E_{n}, n= 0, 1, 2, 
\ldots $. To define a
\textit{Markov measure} $m$, we require that the sequence $(P_n)$
 satisfies the  following properties:
\begin{equation}\label{defn of P_n}
(a)\ \ p^{(n)}_{v,e} > 0 \ \Longleftrightarrow \ (s(e) = v); \ \ \ \
(b)\ \  \sum_{e : s(e) = v} p^{(n)}_{v,e} =1.
\end{equation}
Condition \eqref{defn of P_n}(a) shows that $p^{(n)}_{v,e}$ is 
positive 
only on the edges outgoing from the vertex $v$, and therefore the
 matrices $P_n$ and $A_n =F_n^T$ share the same set of zero entries.
 For any cylinder set $[\overline e] = [(e_0, e_1, \ldots , e_n)]$
generated by the path $\ol e$ with $v =s(e_0) \in V_0$, we set
\begin{equation}\label{eq m([e])}
m([\overline e]) = q_{s(e_0)}p^{(0)}_{s(e_0), e_0}\ \cdots \ 
p^{(n)}_{s(e_n), e_n}.
\end{equation}
Relation \eqref{eq m([e])} defines the value of the measure $m$ 
of the set $[\ol e]$.
By \eqref{defn of P_n}(b),  this measure satisfies the 
\textit{Kolmogorov consistency condition} and can be extended to the
$\sigma$-algebra of  Borel sets. 
To emphasize that $m$ is generated by a sequence of 
\textit{stochastic matrices}, we will also write $m = m(P_n)$.

If all stochastic matrices $P_n$ are equal to a matrix $P$, then 
the corresponding measure $m(P)$ is \textit{called a stationary Markov
measure.}
\end{definition}

\begin{definition} \label{def HI measure}
Let $B$ be a horizontally stationary generalized Bratteli diagram and $m = m(P_n)$ a Markov measure on the path space $X_B$. We say that $m$ is a \textit{horizontally invariant measure}
if $m([\ol e]) = m([\ol e'])$ for any two parallel cylinder sets
$[\ol e]$ and $[\ol e']$.
\end{definition}

The following result follows immediately from \eqref{eq m([e])} and
Definition \ref{def HI measure}. 

\begin{lemma}\label{lem HI}
(1) For a horizontally stationary generalized Bratteli diagram $B$, 
a Markov measure $m = m(P_n)$ is horizontally stationary on $X_B$
if and only if, for every $n \in \N_0$ and $i, j \in V_n$, 
$p^{(n)}_{i, e} =  p^{(n)}_{j,f}$ where $e$ is parallel to $f$ and 
the initial distribution is a constant vector. 

(2) Every horizontally invariant measure on $B$ is sigma-finite. 
\end{lemma}

It follows from Lemma \ref{lem HI} that, for a Markov measure $m = m(P_n)$, the condition of horizontal invariance can be written in the form:
\be\label{eq4 HI}
p^{(n)}_{s(e), e} =  p^{(n)}_{s(\tau^k(e)),\tau^k(e)},\ \ k \in \Z.
\ee 
Relation $\eqref{eq4 HI}$ shows that horizontally invariant
measures are the measures that are invariant under the horizontal shift
 $\tau$.

Let $m =m(P_n)$ be the Markov measure on a horizontally stationary 
generalized Bratteli diagram such that 
\be\label{eq4 uniform meas}
p^{(n)}_{i, e} = \frac{1}{|s^{-1}(i)|}, \ \  i \in V_n, \ \ n \in \Z.
\ee
We call $m$ the \textit{uniform} Markov measure.

\begin{theorem}
Let $m =m(P_n)$ be a horizontally invariant measure on the path space
of a horizontally stationary generalized Bratteli diagram $B = B(F_n)$.
The measure $m$ is tail invariant if and only if it is uniform.
\end{theorem}

\begin{proof}
Suppose that $m$ is a uniform horizontally invariant measure. 
We recall that the matrices $F_n$ (and $A_n = F_n^T$) have the $ERS(r_n)$ and $ECS(r_n)$
properties where $r_n = |s^{-1}(i)|$, $i \in V_n$.
Define the vectors $p^{(n)}$ associated to the levels $V_n$ as follows: set $p^{(0)} = (\ldots, 1, 1, 1, \ldots)$ and 
$p^{(n)} = (p^{(n)}_i : i \in V_n)$ where 
$$
p^{(n)}_i = \frac{1}{r_0\ \cdots\ r_{n-1}}, \quad i \in V_n. 
$$
Because of the properties of matrices $F_n$ and the fact that 
$p^{(n)}$ is a constant vector, we easily obtain that
$A_n p^{(n+1)} = p^{(n)}$ for all $n$. By Theorem 
\ref{BKMS_measures=invlimits}, we conclude that $m$ is tail 
invariant.

Conversely, suppose that a horizontally invariant Markov measure $m(P_n)$ is tail invariant. Firstly, if $q_i = m([i]), i \in V_0$,
then $q_i$ does not depend on $i$. Here $[i]$ is the set of infinite
paths beginning at $i$. 

For every vertex $i \in V_0$, we enumerate the edges from
the set $s^{-1}(i)$ from left to right, 
$e_1, \ldots, e_{r_0}$. The matrix $P_0$ assigns the probabilities 
to these edges (cylinder sets), $p^{(0)}_{i, e_1}, \ldots, 
p^{(0)}_{i, e_{r_0}}$. Simplify the notation and write 
$p^{(0)}_{1}, \ldots, 
p^{(0)}_{r_0}$ because these probabilities do not depend on $i$.
Take a vertex $k \in V_1$ and consider the set $r^{-1}(k) = \{f_1, \ldots, f_{r_0}\}$ of all edges with range $k$. We use here 
Proposition \ref{Prop:prop_hor_gbd} (5) stating that $|r^{-1}(k)| 
= |s^{-1}(i)|$. If we enumerate the edges from $r^{-1}(k)$ from right to left, then the edges $e_j$ and $f_j$ are parallel, $j = 1, ..., r_0$. Therefore, the $m$-measure of the cylinder set $[f_j]$ equals 
$p^{(0)}_j = m([e_j])$ because $m$ is horizontally invariant. 
Since $m$ is tail invariant, we conclude that 
$$
p^{(0)}_0 = \ldots = p^{(0)}_{r_0} = \frac{1}{r_0}.
$$
This means that the rows of $P_0$ are represented by constant 
vectors with the same entry $r_0^{-1}$.

For the next step, we take the matrix $P_1$ and apply the same approach as we used for the matrix $P_0$. We will see that all entries of $P_1$ are equal to $r_1^{-1}$, etc. 
This shows that $m(P_n)$ is the uniform measure.
\end{proof}

\section{Vershik map on horizontally stationary Bratteli diagram}

\begin{definition}\label{def: stat ordered}
Let $B$ be a horizontally stationary generalized Bratteli diagram. 
We say that an order $\omega$ on $B$ is {\em horizontally stationary}
if for every $n\geq 1$, for every $i$ and $i'$ in $V_n$, the sets
$r^{-1}(i)$ and $r^{-1}(i')$ are identically ordered. 
\end{definition}

This means that the edges $e \in E(j, i)$ and the edges $e' \in 
E(j+k, i+k)$ are labeled by the same numbers.
As an obvious consequence of the definition, we note that $e_{\max} \in r^{-1}(i)$ and
$e'_{\max} \in r^{-1}(i')$ are parallel edges. The same is true 
for minimal edges. Therefore every vertex has exactly one outgoing minimal edge and one outgoing maximal edge.
It is easy to see that the following lemma holds.

\begin{lemma}\label{Lemma:horiz_stat_order}
Let $B$ be a horizontally stationary generalized Bratteli diagram with a horizontally stationary order. Then

(1) Every vertex $i \in V_0$ is the source for a unique minimal infinite path $x_{\min}(i)$
and a unique maximal infinite path $x_{\max}(i)$. 
The sets $X_{\max}$ and $X_{\min}$ are countable.

(2) For $i, i' \in V_0$, the infinite paths $x_{\min}(i)$ and $x_{\min}(i')$ consist 
of pairwise parallel edges. The same holds for any pair of infinite
maximal paths.
\end{lemma}

The following theorem gives necessary and sufficient conditions for a Vershik map on a horizontally stationary ordered Bratteli diagram to be extended to a homeomorphism of $X_B$.

\begin{theorem}\label{Thm:continVmap}
Let $B(\omega)$ be a horizontally stationary generalized Bratteli diagram with a horizontally stationary order $\omega$. Then the order $\omega$ 
defines a Vershik homeomorphism $\varphi_B$ if and only if for all $n$ large enough and for every infinite maximal path $x_{\max} = (x^{(n)}_{\max})$, all non-maximal edges with the source $r(x^{(n-1)}_{\max})$ have successors with the same source $v_{n} \in V_n$ ($v_n$ depends on the choice of $x_{\max}$) and there is a minimal edge between $v_{n}$ and $v_{n+1}$. In the case when $\varphi_B$ cannot be extended to a homeomorphism, every infinite maximal path is a point of discontinuity for any extension of $\varphi_B$.

\end{theorem} 

\begin{proof}
First, we prove the ``only if'' part.
    Fix any $x_{\max} \in X_B$ and assume that $\varphi_B$ is a continuous bijection with $\varphi_B(x_{\max}) = x_{\min}$ for some $x_{\min} = (x_{\min}^{(n)})$. We prove that for all $n$ large enough, all non-maximal edges with the source $r(x^{(n-1)}_{\max})$ have successors with the source $v_{n} = r(x_{\min}^{(n-1)})$. 
    
    Note that if for infinitely many $n \in \mathbb{N}$ there are two non-maximal edges $e^{(n)}_1$, $e^{(n)}_2$ with the sources $s(e^{(n)}_1) = s(e^{(n)}_2) = r(x^{(n-1)}_{\max})$ and such that their successors $e'^{(n)}_1$ and $e'^{(n)}_2$ have different sources, then the Vershik map cannot be continuous. Indeed, in this case, there are two non-maximal infinite paths that coincide with $x_{\max}$ on levels $0, \ldots, n-1$, and then go through the non-maximal edges $e^{(n)}_1$, $e^{(n)}_2$. These paths will be mapped to the neighborhoods of infinite minimal paths that pass through the sources of the successors $s(e'^{(n)}_1)$, $s(e'^{(n)}_2)$. 
    By Lemma~\ref{Lemma:horiz_stat_order}, 
    these minimal paths will also start at different vertices at level $V_0$, and thus they will have a distance $1$ from each other. Hence we cannot choose the image of $x_{\max}$ in a continuous way.

Now assume that there is $N$ such that for all $n > N$, all non-maximal edges with the source $r(x^{(n-1)}_{\max})$ have successors with the same source $v_{n}$.  It follows from horizontal stationarity that the same situation will occur for every vertex of level $n$.
Recall that we assume there is a continuous extension of the Vershik map and $\varphi_B(x_{\max}) = x_{\min}$. Then there is $N_1 > N$ such that, for every $n > N_1$, we have $v_n = s(x_{\min}^{(n)})$. Indeed, if for infinitely many $n$ we had $v_n \neq s(x_{\min}^{(n)})$, then we could choose non-maximal paths arbitrarily close to $x_{\max}$ that would be mapped to neighborhoods of minimal paths different from $x_{\min}$ and which are at distance $1$ from $x_{\min}$. Thus, for all $n$ large enough we have $v_n = s(x_{\min}^{(n)})$, hence there is a minimal edge between $v_{n}$ and $v_{n+1}$. 

Now we prove the ``if'' part. Since for all $n$ large enough there is a minimal edge between $v_{n}$ and $v_{n+1}$, there is a unique infinite minimal path $x_{\min}$ which passes through $v_n$ for all $n$ large enough. Since for all such $n$, the non-maximal edges with the source $r(x^{(n-1)}_{\max})$ have successors with the source $v_{n} = r(x_{\min}^{(n-1)})$, we obtain that the Vershik map $\varphi_B$ is continuous. 
Using the same arguments, it is easy to see that $\varphi_B^{-1}$ is also continuous.

Because of horizontal stationarity,  if $\varphi_B$ is discontinuous for one infinite maximal path, then all infinite maximal paths are points of discontinuity for  $\varphi_B$.
\end{proof}

\begin{remark}
    To check the continuity of $\varphi_B^{-1}$ it is enough to ``reverse'' the order $\omega$ (i.e., to define the reversed order $\omega'$, we set  $e <_{\omega} f\  \Longrightarrow \
    e >_{\omega'} f$) and check continuity of $\varphi_B$ for the reversed order $\omega'$.
\end{remark}

The following examples illustrate the statement of Theorem~\ref{Thm:continVmap}.

\begin{example}[Vershik map can be extended to a  homeomorphism]
The diagram on Figure~\ref{Fig:ContinVMap1} satisfies conditions of Theorem~\ref{Thm:continVmap} for all $n$. Every vertex on every level $n \geq 1$ has three incoming edges, the ordering is left-to-right, and all maximal edges of the diagram are vertical. For every $n \geq 0$, the edges of level $E_n$ determine how to draw the edges of level $E_{n+1}$ in order to satisfy conditions of Theorem~\ref{Thm:continVmap}. For an infinite maximal path that passes through vertex $w$ on each level, its image under the Vershik map is a minimal path that passes through vertices $v_n = w + 3^n$. 

\begin{center}

\begin{figure}[ht]
\unitlength=0,8cm
\begin{tikzpicture}[shorten >=1pt,node distance=2cm,auto]
\tikzstyle{state}=[shape=circle,fill,inner sep=0pt,minimum size=4pt]
\node[state] (A1) {};
\node[state,above of=A1] (B1) {}; 
\node[state,above of=B1] (C1) {};
\node[state,right of=A1] (A2) {};
\node[state,right of=A2] (A3) {};
\node[state,right of=A3] (A4) {};
\node[state,right of=A4] (A5) {};
\node[state,right of=A5] (A6) {};
\node[state,right of=A6] (A7) {};
\node[state,above of=A2] (B2) {};
\node[state,above of=A3] (B3) {};
\node[state,above of=A4] (B4) {};
\node[state,above of=A5] (B5) {};
\node[state,above of=A6] (B6) {};
\node[state,above of=A7] (B7) {};
\node[state,above of=B2] (C2) {};
\node[state,above of=B3] (C3) {};
\node[state,above of=B4] (C4) {};
\node[state,above of=B5] (C5) {};
\node[state,above of=B6] (C6) {};
\node[state,above of=B7] (C7) {};
\path[-,draw,thick]
  (A1) edge node[] {} (B1)
  (A2) edge node[] {} (B2)
  (A3) edge node[] {} (B3)
  (A4) edge node[] {} (B4)
  (A5) edge node[] {} (B5)
  (A6) edge node[] {} (B6)
  (A7) edge node[] {} (B7)
  (C1) edge node[] {} (B1)
  (C2) edge node[] {} (B2)
  (C3) edge node[] {} (B3)
  (C4) edge node[] {} (B4)
  (C5) edge node[] {} (B5)
  (C6) edge node[] {} (B6)
  (C7) edge node[] {} (B7)
  (B2) edge node[] {} (C1)
  (B3) edge node[] {} (C2)
  (B4) edge node[] {} (C3)
  (B5) edge node[] {} (C4)
  (B6) edge node[] {} (C5)
  (B7) edge node[] {} (C6)
  (B3) edge node[] {} (C1)
  (B4) edge node[] {} (C2)
  (B5) edge node[] {} (C3)
  (B6) edge node[] {} (C4)
  (B7) edge node[] {} (C5)
  (A4) edge node[] {} (B1)
  (A5) edge node[] {} (B2)
  (A6) edge node[] {} (B3)
  (A7) edge node[] {} (B4)
  (A7) edge node[] {} (B1)
  (A3) -- ($ (A1) !.66! (B1) $)
  (B5) -- ($ (B7) !.66! (A7) $)
  (B6) -- ($ (B7) !.33! (A7) $)
  (A2) -- ($ (A1) !.33! (B1) $)
  (C6) -- ($ (C7) !.5! (B7) $)
  (A6) -- ($ (A1) !.8! (B1) $)  
  (A5) -- ($ (A1) !.6! (B1) $)  
  (A4) -- ($ (A1) !.4! (B1) $)
  (A3) -- ($ (A1) !.2! (B1) $)
  (B2) -- ($ (A7) !.2! (B7) $)  
  (B3) -- ($ (A7) !.4! (B7) $)  
  (B4) -- ($ (A7) !.6! (B7) $)
  (B5) -- ($ (A7) !.8! (B7) $)
  (B2) -- ($ (B1) !.5! (C1) $);
\draw (0,4.4) node {$w$};
\draw (2,4.4) node {$v_0$};
\draw (6.3,2.2) node {$v_1$};
\draw (12.6,4) node {$\ldots$};
\draw (12.6,2) node {$\ldots$};
\draw (12.6,0) node {$\ldots$};
\draw (-0.6,4) node {$\ldots$};
\draw (-0.6,2) node {$\ldots$};
\draw (-0.6,0) node {$\ldots$};
\draw (0,-0.5) node {$\vdots$};
\draw (2,-0.5) node {$\vdots$};
\draw (4,-0.5) node {$\vdots$};
\draw (6,-0.5) node {$\vdots$};
\draw (8,-0.5) node {$\vdots$};
\draw (10,-0.5) node {$\vdots$};
\draw (12,-0.5) node {$\vdots$};
\draw (1.4,2.1) node {$0$};
\draw (1.8,2.4) node {$1$};
\draw (2.2,2.5) node {$2$};
\draw (3.4,2.1) node {$0$};
\draw (3.8,2.4) node {$1$};
\draw (4.2,2.5) node {$2$};
\draw (5.4,2.1) node {$0$};
\draw (5.8,2.4) node {$1$};
\draw (6.2,2.5) node {$2$};
\draw (7.4,2.1) node {$0$};
\draw (7.8,2.4) node {$1$};
\draw (8.2,2.5) node {$2$};
\draw (9.4,2.1) node {$0$};
\draw (9.8,2.4) node {$1$};
\draw (10.2,2.5) node {$2$};
\draw (11.4,2.1) node {$0$};
\draw (11.8,2.4) node {$1$};
\draw (12.2,2.5) node {$2$};
\end{tikzpicture}
\caption{A continuous Vershik map on a horizontally stationary Bratteli diagram (left-to-right 
ordering, each vertex has three incoming edges, vertical edges are maximal)}\label{Fig:ContinVMap1}
\end{figure}
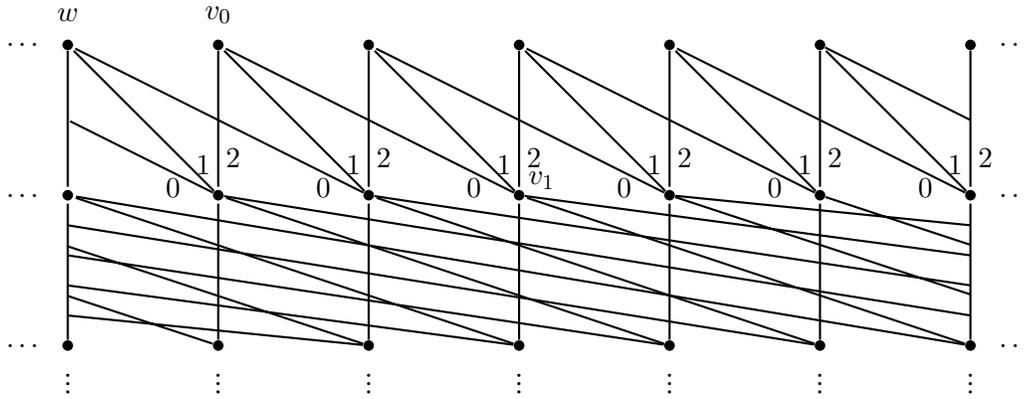
\end{center}

\end{example}

Another example of a diagram satisfying conditions of Theorem~\ref{Thm:continVmap} can be found in Example 3.13 of \cite{BezuglyiJorgensenKarpelSanadhya2023}. There every vertex from $V \setminus V_0$ has exactly two incoming edges.

\begin{example} This example shows that all conditions of Theorem 
\ref{Thm:continVmap} are important. 
Consider a vertically stationary and horizontally stationary generalized Bratteli diagram with the left-to-right order and incidence matrix $F = (f_{ij})$ given by
$$
f_{ij} = 
\left\{
\begin{aligned}
& 1, \mbox{ for } |i - j| \leq 1,\\
& 0, \mbox{ otherwise. }
\end{aligned}
\right.
$$
It is easy to see that that one of the conditions of Theorem~\ref{Thm:continVmap} is satisfied: for all $n$ and for every infinite maximal path $x_{\max} = (x^{(n)}_{\max})$, all non-maximal edges with the source $r(x^{(n-1)}_{\max})$ have successors with the same source $v_{n} = (r(x^{(n-1)}_{\max}) + 1) \in V_n$. But the second condition is not satisfied since there is only a maximal edge between vertices $v_{n} \in V_n$ and $v_{n+1} = (v_n - 1) \in V_{n+1}$. Hence Vershik map is not continuous for any extension of $\varphi_B$ to the whole $X_B$.
\end{example}

\medskip
\textbf{Acknowledgments.} 
We are very grateful to our colleagues, 
especially, H. Bruin, S. Radinger, T. Raszeja, S. Sanadhya, M. Wata for the numerous valuable discussions. We are also very grateful to the referee for the detailed comments which helped us to improve the exposition of the paper.
O.K. is supported by the NCN (National Science Centre, Poland) Grant 2019/35/D/ST1/01375 and by the program ``Excellence Initiative - Research University'' for the AGH University of Krakow.

\bibliographystyle{alpha}
\bibliography{bibliographyPascal}

\end{document}